\documentclass[a4paper,11pt]{amsart}
\usepackage{mathrsfs}
\usepackage{appendix}
\usepackage{amssymb} 
\usepackage{fancyhdr}

\makeatletter
      \def\@setcopyright{}
      \def\serieslogo@{}
      \makeatother

\newcommand{\Complex}{\mathbb C}
\newcommand{\Real}{\mathbb R}
\newcommand{\N}{\mathbb N}
\newcommand{\ddbar}{\overline\partial}
\newcommand{\pr}{\partial}
\newcommand{\ol}{\overline}
\newcommand{\Td}{\widetilde}
\newcommand{\norm}[1]{\left\Vert#1\right\Vert}
\newcommand{\abs}[1]{\left\vert#1\right\vert}

\newcommand{\set}[1]{\left\{#1\right\}}
\newcommand{\To}{\rightarrow}

\theoremstyle{plain}
\newtheorem{thm}{Theorem}[section]
\newtheorem{cor}[thm]{Corollary}

\newtheorem{lem}[thm]{Lemma}
\newtheorem{prop}[thm]{Proposition}

\theoremstyle{definition}
\newtheorem{defn}[thm]{Definition}

\theoremstyle{remark}
\newtheorem{rem}[thm]{Remark}

\numberwithin{equation}{section}

\begin{document}
\title[]{On CR Paneitz operators and CR pluriharmonic functions}
\author[]{Chin-Yu Hsiao}
\address{Institute of Mathematics, Academia Sinica, 6F, Astronomy-Mathematics Building, No.1, Sec.4, Roosevelt Road, Taipei 10617, Taiwan}
\email{chsiao@math.sinica.edu.tw or chinyu.hsiao@gmail.com} 

\begin{abstract}
Let $(X,T^{1,0}X)$ be a compact orientable embeddable three dimensional strongly pseudoconvex CR manifold and let ${\rm P\,}$ be the associated CR Paneitz operator. In this paper, we show that (I) ${\rm P\,}$ is self-adjoint and ${\rm P\,}$ has $L^2$ closed range. Let $N$ and $\Pi$ be the associated partial inverse and the orthogonal projection onto ${\rm Ker\,}{\rm P\,}$ respectively, then $N$ and $\Pi$ enjoy some regularity properties. 
(II) Let $\hat{\mathcal{P}}$ and $\hat{\mathcal{P}_0}$ be the space of $L^2$ CR pluriharmonic functions and the space of real part of $L^2$ global CR functions respectively. Let $S$ be the associated Szeg\"o projection and let $\tau$, $\tau_0$ be the orthogonal projections onto $\hat{\mathcal{P}}$ and $\hat{\mathcal{P}_0}$ respectively. Then, $\Pi=S+\ol S+F_0$, $\tau=S+\ol S+F_1$, $\tau_0=S+\ol S+F_2$, where $F_0, F_1, F_2$ are smoothing operators on $X$. In particular, $\Pi$, $\tau$ and $\tau_0$ are Fourier integral operators with complex phases and $\hat{\mathcal{P}}^\perp\bigcap{\rm Ker\,}{\rm P\,}$, $\hat{\mathcal{P}_0}^\perp\bigcap\hat{\mathcal{P}}$, $\hat{\mathcal{P}_0}^\perp\bigcap{\rm Ker\,}{\rm P\,}$ are all finite dimensional subspaces of $C^\infty(X)$ (it is well-known that $\hat{\mathcal{P}_0}\subset\hat{\mathcal{P}}\subset{\rm Ker\,}{\rm P\,}$). (III) ${\rm Spec\,}{\rm P\,}$ is a discrete subset of $\Real$ and for every $\lambda\in{\rm Spec\,}{\rm P\,}$, $\lambda\neq0$, $\lambda$ is an eigenvalue of ${\rm P\,}$ and the associated eigenspace $H_\lambda({\rm P\,})$ is a finite dimensional subspace of $C^\infty(X)$. 
\end{abstract} 

\maketitle \tableofcontents

\section{Introduction and statement of the main results} \label{s-intro}

Let $(X,T^{1,0}X)$ be a compact orientable embeddable strongly pseudoconvex CR manifold of dimension three. Let ${\rm P\,}$ be the associated Paneitz operator and let $\hat{\mathcal{P}}$ be the space of $L^2$ CR pluriharmonic functions. The operator ${\rm P\,}$ and the space $\hat{\mathcal{P}}$ play important roles in CR embedding problems and CR conformal geometry (see~\cite{CY13}~\cite{CCY10},~\cite{CCY12}). The operator 
\[{\rm P\,}:{\rm Dom\,}{\rm P\,}\subset L^2(X)\To L^2(X)\] 
is a real, symmetric, fourth order non-hypoelliptic partial differential operator and $\hat{\mathcal{P}}$ is an infinite dimensional subspace of $L^2(X)$.  In CR embedding problems and CR conformal geometry, it is crucial to be able to answer the following fundamental analytic problems about ${\rm P\,}$ and $\hat{\mathcal{P}}$ (see~\cite{CY13}~\cite{CCY10},~\cite{CCY12}): 

(I) Is ${\rm P\,}$ self-adjoint? Does ${\rm P\,}$ has $L^2$ closed range? What is ${\rm Spec\,}{\rm P\,}$ ? 

(II) If we have ${\rm P\,}u=f$, where $f$ is in some Sobolev space $H^s(X)$, $s\in\mathbb Z$, and $u\perp{\rm Ker\,}{\rm P\,}$. Can we have $u\in H^{s'}(X)$, for some $s'\in\mathbb Z$? 

(III) It is well-known (see Lee~\cite{Lee88}) that  $\hat{\mathcal{P}}\subset{\rm Ker\,}{\rm P\,}$ and if $X$ has torsion zero then $\hat{\mathcal{P}}={\rm Ker\,}{\rm P\,}$. It remains an important problem to determine the precise geometrical condition under which the kernel of ${\rm P\,}$ is exactly the CR pluriharmonic functions or even a direct sum of a finite dimensional subspace with CR pluriharmonic functions. 

(IV) Let $\Pi$ be the orthogonal projection onto ${\rm Ker\,}{\rm P\,}$ and let $\tau$ be the orthogonal projection onto $\hat{\mathcal{P}}$. Let $\Pi(x,y)$ and $\tau(x,y)$ denote the distribution kernels of $\Pi$ and $\tau$ respectively.  The ${\rm P'\,}$ operator introduced in Case and Yang~\cite{CY13} plays a critical role in CR conformal geometry. To understand the operator 
${\rm P'\,}$, it is crucial to be able to know the exactly forms of $\Pi(x,y)$ and $\tau(x,y)$.

The purse of this work is to completely answer this questions. On the other hand, in several complex variables, the study of the associated Szeg\"o projection $S$ and $\tau$ are classical subjects. The operator $S$ is well-understood; $S$ is a Fourier integral operator with complex phase (see Boutet de Monvel-Sj\"ostrand~\cite{BouSj76},~\cite{Hsiao08},~\cite{HM14}). But for $\tau$, there are fewer results.   In this paper, by using the Paneitz operator ${\rm P\,}$, we could prove that $\tau$ is also a complex Fourier integral operator and $\tau=S+\ol S+F_1$, $F_1$ is a smoothing operator. It is quite interesting to see if the result hold in dimension$\geq5$.  We hope that the Paneitz operator ${\rm P\,}$ will be interesting for complex analysts and will be useful in several complex variables. 

We now formulate the main results. We refer to section~\ref{s:prelim} for some standard notations and terminology used here.

Let $(X,T^{1,0}X)$ be a compact orientable $3$-dimensional strongly pseudoconvex CR manifold, 
where $T^{1,0}X$ is a CR structure of $X$. We assume throughout that it is CR embeddable in some $\mathbb{C}^N$, for some $N\in\mathbb N$. Fix a contact form $\theta\in C^\infty(X,T^*X)$ compactable with the CR structure $T^{1,0}X$. Then, $(X,T^{1,0}X,\theta)$ is a $3$-dimensional pseudohermitian manifold. Let $T\in C^\infty(X,TX)$ be the real non-vanishing global vector field given by 
\[\begin{split}
&\langle\,d\theta\,,\,T\wedge u\,\rangle=0,\ \ \forall u\in T^{1,0}X\oplus T^{0,1}X,\\
&\langle\,\theta\,,\,T\,\rangle=-1.
\end{split}\]
Let $\langle\,\cdot\,|\,\cdot\,\rangle$ be the Hermitian inner product on $\Complex TX$ given by 
\[\begin{split}
&\langle Z_1 | Z_2 \rangle=-\frac{1}{2i}\langle\,d\theta\,,\,Z_1\wedge\overline{Z}_2\,\rangle, Z_1, Z_2\in T^{1,0}X,\\
&T^{1,0}X\perp T^{0,1}X:=\ol{T^{1,0}X},\ \ T\perp(T^{1,0}X\oplus T^{0,1}X),\ \ \langle\,T\,|\,T\,\rangle=1.
\end{split}\]
The Hermitian metric $\langle\,\cdot\,|\,\cdot\,\rangle$ on $\Complex TX$ induces a Hermitian metric $\langle\,\cdot\,|\,\cdot\,\rangle$ on $\Complex T^*X$. Let $T^{*0,1}X$ be the bundle of $(0,1)$ forms of $X$. Take $\theta\wedge d\theta$ be the volume form on $X$, we then get natural inner products on $C^\infty(X)$ and $\Omega^{0,1}(X):=C^\infty(X,T^{*0,1}X)$ induced by $\theta\wedge d\theta$ and $\langle\,\cdot\,|\,\cdot\,\rangle$. We shall use $(\,\cdot\,|\,\cdot\,)$ to denote these inner products and use $\norm{\cdot}$ to denote the corresponding norms. Let $L^2(X)$ and $L^2_{(0,1)}(X)$ denote the completions of $C^\infty(X)$ and $\Omega^{0,1}(X)$ with respect to $(\,\cdot\,|\,\cdot\,)$ respectively. Let 
\[\Box_b:=\ddbar^{*,f}_b\ddbar_b:C^\infty(X)\To C^\infty(X)\]
be the Kohn Laplacian (see~\cite{Hsiao08}), where $\ddbar_b:C^\infty(X)\To\Omega^{0,1}(X)$ is the tangential Cauchy-Riemann operator and $\ddbar^{*,f}_b:\Omega^{0,1}(X)\To C^\infty(X)$ is the formal adjoint of $\ddbar_b$ with respect to $(\,\cdot\,|\,\cdot\,)$. That is, $(\,\ddbar_bf\,|\,g\,)=(\,f\,|\,\ddbar^{*,f}_bg\,)$, for every $f\in C^\infty(X)$, $g\in\Omega^{0,1}(X)$. 

Let $\mathcal P$ be the set of all CR pluriharmonic functions on $X$. That is, 
\begin{equation}\label{e-gue140325III}
\begin{split}
\mathcal P=&\{u\in C^\infty(X,\Real);\, \forall x_0\in X, \mbox{there is a $f\in C^\infty(X)$}\\
&\mbox{with $\ddbar_bf=0$ near $x_0$ and ${\rm Re\,}f=u$ near $x_0$}\}. 
\end{split}
\end{equation} 
The Paneitz operator
\[{\rm P\,}:C^\infty(X)\To C^\infty(X)\]
can be characterized as follows (see section 4 in~\cite{CY13} and Lee~\cite{Lee88}): ${\rm P\,}$ is a fourth order partial differential operator, real, symmetric, $\mathcal{P}\subset{\rm Ker\,}{\rm P\,}$ and
\begin{equation}\label{e-gue140325}
\begin{split}
&{\rm P\,}f=\Box_b\ol\Box_bf+L_1\circ L_2f+L_3f,\ \ \forall f\in C^\infty(X),\\
&L_1, L_2, L_3\in C^\infty(X,T^{1,0}X\oplus T^{0,1}X).
\end{split}
\end{equation}
We extend ${\rm P\,}$ to $L^2$ space by 
\begin{equation}\label{e-gue140325I}
\begin{split}
&{\rm P\,}:{\rm Dom\,}{\rm P\,}\subset L^2(X)\To L^2(X),\\
&{\rm Dom\,}{\rm P\,}=\set{u\in L^2(X);\, {\rm P\,}u\in L^2(X)}.
\end{split}
\end{equation}

Let $\hat{\mathcal{P}}\subset L^2(X)$ be the completion of 
$\mathcal P$ with respect to $(\,\cdot\,|\,\cdot\,)$. Then, 
\[\hat{\mathcal P}\subset{\rm Ker\,}{\rm P\,}.\] Put 
\[
\mathcal P_0=\{{\rm Re\,}f\in C^\infty(X,\Real);\, \mbox{$f\in C^\infty(X)$ is a global CR function on $X$}\}
\]
and let $\hat{\mathcal P_0}\subset L^2(X)$ be the completion of $\mathcal P_0$ with respect to $(\,\cdot\,|\,\cdot\,)$. It is clearly that $\hat{\mathcal P_0}\subset\hat{\mathcal P}\subset{\rm Ker\,}{\rm P\,}$. Let 
\begin{equation}\label{e-gue140325IV}
\begin{split}
&\tau:L^2(X)\To\hat{\mathcal{P}},\\
&\tau_0:L^2(X)\To\hat{\mathcal{P}_0},
\end{split}
\end{equation}
be the orthogonal projections. 

We recall

\begin{defn} \label{d-gue140211}
Suppose $Q$ is a closed densely defined self-adjoint operator
\[Q:{\rm Dom\,}Q\subset H\To{\rm Ran\,}Q\subset H,\]
where $H$ is a Hilbert space. Suppose that $Q$ has closed range. By the partial inverse of $Q$, we
mean the bounded operator $M: H\To {\rm Dom\,}Q$ such that
\[\begin{split}
&\mbox{$QM+\pi=I$ on $H$},\\
&\mbox{$MQ+\pi=I$ on ${\rm Dom\,}Q$},
\end{split}\]
where $\pi:H\To{\rm Ker\,}Q$ is the orthogonal projection.
\end{defn}

The main purpose of this work is to prove the following 

\begin{thm}\label{t-gue140325} 
With the notations and assumptions above, 
\[{\rm P\,}:{\rm Dom\,}{\rm P\,}\subset L^2(X)\To L^2(X)\]
is self-adjoint and ${\rm P\,}$ has $L^2$ closed range. Let $N:L^2(X)\To{\rm Dom\,}{\rm P\,}$ be the partial inverse and let $\Pi:L^2(X)\To{\rm Ker\,}{\rm P\,}$ be the orthogonal projection. Then, 
\begin{equation}\label{e-gue140325V}
\begin{split}
&\mbox{$\Pi, \tau, \tau_0:H^s(X)\To H^s(X)$ is continuous, $\forall s\in\mathbb Z$},\\
&\mbox{$N:H^s(X)\To H^{s+2}(X)$ is continuous, $\forall s\in\mathbb Z$},
\end{split}
\end{equation}
\begin{equation}\label{e-gue140325VI}
\mbox{$\Pi\equiv\tau$ on $X$, $\Pi\equiv\tau_0$ on $X$}
\end{equation}
and the kernel $\Pi(x,y)\in\mathscr D'(X\times X)$ of $\Pi$ satisfies 
\begin{equation}\label{e-gue140325VII}
\Pi(x, y)\equiv\int^{\infty}_{0}e^{i\varphi(x, y)t}a(x, y, t)dt+\int^{\infty}_{0}e^{-i\ol\varphi(x, y)t}\ol a(x, y, t)dt,
\end{equation}
where
\begin{equation}\label{e-gue140325VIII}
\begin{split}
&\varphi\in C^\infty(X\times X),\ \ {\rm Im\,}\varphi(x, y)\geq0,\ \ d_x\varphi|_{x=y}=-\theta(x),\\
&\varphi(x,y)=-\ol\varphi(y,x),\\
&\mbox{$\varphi(x,y)=0$ if and only if $x=y$},
\end{split}
\end{equation}
(see Theorem~\ref{t-gue140305I} and Theorem~\ref{t-gue140305II} for more properties of the phase $\varphi$), and 
\begin{equation}\label{e-gue140325aI}\begin{split}
&a(x, y, t)\in S^{1}_{{\rm cl\,}}(X\times X\times]0, \infty[), \\
&\mbox{$a(x, y, t)\sim\sum^\infty_{j=0}a_j(x, y)t^{1-j}$ in $S^1_{1,0}(X\times X\times]0,\infty[)$},\\
&a_j(x,y)\in C^\infty(X\times X),\ \ j=0,1,\ldots,\\
&a_0(x,x)=\frac{1}{2}\pi^{-n},\ \ \forall x\in X.
\end{split}\end{equation}
(See section~\ref{s:prelim} and Definition~\ref{d-gue140329} for the precise meanings of the notation $\equiv$ and the H\"ormander symbol spaces $S^1_{{\rm cl\,}}(X\times X\times]0,\infty[)$ and $S^1_{1,0}(X\times X\times]0,\infty[)$.
\end{thm}

\begin{rem}\label{r-gue140329}
With the notations and assumptions used in Theorem~\ref{t-gue140325}, it is easy to see that $\Pi$ is real, that is $\Pi=\ol\Pi$. 
\end{rem}

\begin{rem}\label{r-gue140325}
With the notations and assumptions used in Theorem~\ref{t-gue140325}, let $S:L^2(X)\To{\rm Ker\,}\ddbar_b$ be the Szeg\"o projection. That is, $S$ is the orthogonal projection onto ${\rm Ker\,}\ddbar_b=\set{u\in L^2(X);\, \ddbar_bu=0}$ with respect to $(\,\cdot\,|\,\cdot\,)$. In view of the proof of Theorem~\ref{t-gue140325} (see section~\ref{s-gue140326}), we see that $\Pi\equiv S+\ol S$ on $X$. 
\end{rem}

We have the classical formulas
\begin{equation} \label{e:0}
\int^\infty_0\!\!e^{-tx}t^mdt= 
\begin{cases}
m!x^{-m-1}, &\text{if } m\in\mathbb Z,\ \ m\geq 0,  \\
\frac{(-1)^m}{(-m-1)!}x^{-m-1}(\log x+c-\sum^{-m-1}_1\frac{1}{j}), &\text{if } m\in\mathbb Z,\ \ m<0.
\end{cases}\end{equation}
Here $x\neq0$, ${\rm Re\,}x\geq0$ and $c$ is the Euler constant, i.e.
$c=\lim_{m\To\infty}(\sum^m_1\frac{1}{j}-\log m)$.
Note that
\begin{equation} \label{e:0809061310}
\int^{\infty}_0\!\!e^{i\varphi(x, y)t}\sum^\infty_{j=0}a_j(x, y)t^{1-j}dt =\lim_{\varepsilon\to0+}
\int^{\infty}_0\!\!e^{-t\bigr(-i(\varphi(x, y)+i\varepsilon)\bigr)}\sum^\infty_{j=0}a_j(x, y)t^{1-j}dt.
\end{equation}
We have the following corollary of Theorem~\ref{t-gue140325}

\begin{cor} \label{c:1}
With the notations and assumptions used in Theorem~\ref{t-gue140325}, 
there exist $F_1, G_1,\in C^\infty(X\times X)$
such that
\[\begin{split}
\Pi(x,y)=&F_1(-i\varphi(x, y))^{-2}+G_1\log(-i\varphi(x, y))\\
&+\ol F_1(i\ol\varphi(x, y))^{-2}+\ol G_1\log(i\ol\varphi(x, y)).\end{split}\]

Moreover, we have
\begin{equation} \label{e:000} \begin{split}
F_1 &=a_0(x, y)+a_1(x,y)(-i\varphi(x, y))+f_1(x, y)(-i\varphi(x, y))^2, \\
G_1 &\equiv\sum^\infty_0\frac{(-1)^{k+1}}{k!}a_{2+k}(x, y)(-i\varphi(x, y))^k, 
\end{split}\end{equation}
where $a_j(x,y)$, $j=0,1,\ldots$, are as in \eqref{e-gue140325aI} and $f_1(x, y)\in C^\infty(X\times X)$. 
\end{cor}

Put 
\[\begin{split}
&\hat{\mathcal{P}}^\perp:=\set{u\in L^2(X);\, (\,u\,|\,f\,)=0, \forall f\in\hat{\mathcal{P}}},\\  &\hat{\mathcal{P}_0}^\perp:=\set{v\in L^2(X);\, (\,v\,|\,g\,)=0, \forall g\in\hat{\mathcal{P}_0}}.\end{split}\] 
From \eqref{e-gue140325VI} and some standard argument in functional analysis (see section~\ref{s-gue140326}), we deduce 

\begin{cor}\label{c-gue140325}
With the notations and assumptions above, 
we have 
\[\hat{\mathcal{P}}^\perp\bigcap{\rm Ker\,}{\rm P\,}\subset C^\infty(X),\ \ \hat{\mathcal{P}_0}^\perp\bigcap{\rm Ker\,}{\rm P\,}\subset C^\infty(X),\ \ \hat{\mathcal{P}_0}^\perp\bigcap\hat{\mathcal{P}}\subset C^\infty(X)\] 
and $\hat{\mathcal{P}}^\perp\bigcap{\rm Ker\,}{\rm P\,}$, $\hat{\mathcal{P}_0}^\perp\bigcap{\rm Ker\,}{\rm P\,}$, $\hat{\mathcal{P}_0}^\perp\bigcap\hat{\mathcal{P}}$ are all finite dimensional. 
\end{cor}

We have the orthogonal decompositions 
\begin{equation}\label{e-gue140325aII}
\begin{split}
&{\rm Ker\,}{\rm P\,}=\hat{\mathcal{P}}^\perp\oplus(\hat{\mathcal{P}}^\perp\bigcap{\rm Ker\,}{\rm P\,}),\\
&{\rm Ker\,}{\rm P\,}=\hat{\mathcal{P}_0}^\perp\oplus(\hat{\mathcal{P}_0}^\perp\bigcap{\rm Ker\,}{\rm P\,}),\\
&\hat{\mathcal{P}}=\hat{\mathcal{P}_0}\oplus(\hat{\mathcal{P}_0}^\perp\bigcap\hat{\mathcal{P}}).
\end{split}
\end{equation}
From Corollary~\ref{c-gue140325}, we know that $\hat{\mathcal{P}}^\perp\bigcap{\rm Ker\,}{\rm P\,}$, $\hat{\mathcal{P}_0}^\perp\bigcap{\rm Ker\,}{\rm P\,}$, $\hat{\mathcal{P}_0}^\perp\bigcap\hat{\mathcal{P}}$ are all finite dimensional subsets of $C^\infty(X)$. 

Since ${\rm P\,}$ is self-adjoint, ${\rm Spec\,}{\rm P\,}\subset\Real$. In section~\ref{s-gue140413}, we establish spectral theory for ${\rm P\,}$. 

\begin{thm}\label{t-gue140325I}
With the notations and assumptions above, 
${\rm Spec\,}{\rm P\,}$ is a discrete subset in $\Real$ and for every $\lambda\in{\rm Spec\,}{\rm P\,}$, $\lambda\neq0$, $\lambda$ is an eigenvalue of ${\rm P\,}$ and  the eigenspace 
\[H_\lambda({\rm P\,}):=\set{u\in{\rm Dom\,}{\rm P\,};\, {\rm P\,}u=\lambda u}\]
is a finite dimensional subspace of $C^\infty(X)$. 
\end{thm}

\subsection{The phase $\varphi$}\label{s-gue140325}

In this section, we collect some properties of the phase function $\varphi$. We refer the reader to~\cite{Hsiao08} and~\cite{HM14} for the proofs. 

The following result describes the phase function $\varphi$ in local coordinates.

\begin{thm} \label{t-gue140305I}
With the assumptions and notations used in Theorem~\ref{t-gue140325}, 
for a given point $x_0\in X$, let $\{Z_1\}$
be an orthonormal frame of $T^{1,0}X$ in a neighbourhood of $x_0$,  i.e.\ $\mathcal{L}_{x_{0}}(Z_{1},\overline{Z}_{1})=1$.
Take local coordinates $x=(x_1,x_2,x_3)$, $z=x_1+ix_2$, 
defined on some neighbourhood of $x_0$ such that $\theta(x_0)=dx_{3}$, $x(x_0)=0$, and for some $c\in\Complex$, 
\[Z_1=\frac{\pr}{\pr z}-i\ol z\frac{\pr}{\pr x_{3}}-
cx_{3}\frac{\pr}{\pr x_{3}}+O(\abs{x}^2).\]
Set
$y=(y_1,y_2,y_3)$, $w=y_{1}+iy_2$.
Then, for $\varphi$ in Theorem~\ref{t-gue140325}, we have
\begin{equation} \label{e-gue140205VIm}
{\rm Im\,}\varphi(x,y)\geq c\sum^{2}_{j=1}\abs{x_j-y_j}^2,\ \ c>0,
\end{equation}
in some neighbourhood of $(0,0)$ and
\begin{equation} \label{e-gue140205VIIm}
\begin{split}
&\varphi(x, y)=-x_{3}+y_{3}+i\abs{z-w}^2 \\
&\quad+\Bigr(i(\ol zw-z\ol w)+c(-zx_{3}+wy_{3})\\
&\quad+\ol c(-\ol zx_{3}+\ol wy_{3})\Bigr)+(x_{3}-y_{3})f(x, y) +O(\abs{(x, y)}^3),
\end{split}
\end{equation}
where $f$ is smooth and satisfies $f(0,0)=0$, $f(x, y)=\ol f(y, x)$.
\end{thm}

\begin{defn}\label{d-gue140305}
With the assumptions and notations used in Theorem~\ref{t-gue140325}, let $\varphi_1(x,y), \varphi_2(x,y)\in C^\infty(X\times X)$. We assume that $\varphi_1(x,y)$ and $\varphi_2(x,y)$ satisfy 
\eqref{e-gue140325VIII} and \eqref{e-gue140205VIm}. We say that $\varphi_1(x,y)$ and $\varphi_2(x,y)$ are equivalent on $X$ if for any  $b_1(x,y,t)\in  S^{1}_{{\rm cl\,}}(X\times X\times]0, \infty[)$
we can find $b_2(x,y,t)\in  S^{1}_{{\rm cl\,}}(X\times X\times]0, \infty[)$
such that 
\[\int^\infty_0e^{i\varphi_1(x,y)t}b_1(x,y,t)dt\equiv e^{i\varphi_2(x,y)t}b_2(x,y,t)dt\ \ \mbox{on $X$}\]
and vise versa. 
\end{defn} 

We characterize the phase $\varphi$

\begin{thm} \label{t-gue140305II}
With the assumptions and notations used in Theorem~\ref{t-gue140325}, let $\varphi_1(x,y)\in C^\infty(X\times X)$. We assume that $\varphi_1(x,y)$ satisfies \eqref{e-gue140325VIII} and \eqref{e-gue140205VIm}. $\varphi_1(x,y)$ and $\varphi(x,y)$ are equivalent on $X$ in the sense of Definition~\ref{d-gue140305} if and only if there is a function $h\in C^\infty(X\times X)$ such that $\varphi_1(x,y)-h(x,y)\varphi(x,y)$
vanishes to infinite order at $x=y$, for every $(x,x)\in X\times X$. 
\end{thm}

\section{Preliminaries}\label{s:prelim}

We shall use the following notations: $\Real$ is the set of real numbers, $\ol\Real_+:=\set{x\in\Real;\, x\geq0}$, $\mathbb N=\set{1,2,\ldots}$, $\mathbb N_0=\mathbb N\bigcup\set{0}$. An element $\alpha=(\alpha_1,\ldots,\alpha_n)$ of $\mathbb N_0^n$ will be called a multiindex, the size of $\alpha$ is: $\abs{\alpha}=\alpha_1+\cdots+\alpha_n$ and the length of $\alpha$ is $l(\alpha)=n$. For $m\in\mathbb N$, we write $\alpha\in\set{1,\ldots,m}^n$ if $\alpha_j\in\set{1,\ldots,m}$, $j=1,\ldots,n$. We say that $\alpha$ is strictly increasing if $\alpha_1<\alpha_2<\cdots<\alpha_n$. We write $x^\alpha=x_1^{\alpha_1}\cdots x^{\alpha_n}_n$, 
$x=(x_1,\ldots,x_n)$,
$\pr^\alpha_x=\pr^{\alpha_1}_{x_1}\cdots\pr^{\alpha_n}_{x_n}$, $\pr_{x_j}=\frac{\pr}{\pr x_j}$, $\pr^\alpha_x=\frac{\pr^{\abs{\alpha}}}{\pr x^\alpha}$, $D^\alpha_x=D^{\alpha_1}_{x_1}\cdots D^{\alpha_n}_{x_n}$, $D_x=\frac{1}{i}\pr_x$, $D_{x_j}=\frac{1}{i}\pr_{x_j}$. 
Let $z=(z_1,\ldots,z_n)$, $z_j=x_{2j-1}+ix_{2j}$, $j=1,\ldots,n$, be coordinates of $\Complex^n$.  
We write $z^\alpha=z_1^{\alpha_1}\cdots z^{\alpha_n}_n$, $\ol z^\alpha=\ol z_1^{\alpha_1}\cdots\ol z^{\alpha_n}_n$,
$\frac{\pr^{\abs{\alpha}}}{\pr z^\alpha}=\pr^\alpha_z=\pr^{\alpha_1}_{z_1}\cdots\pr^{\alpha_n}_{z_n}$, $\pr_{z_j}=
\frac{\pr}{\pr z_j}=\frac{1}{2}(\frac{\pr}{\pr x_{2j-1}}-i\frac{\pr}{\pr x_{2j}})$, $j=1,\ldots,n$. 
$\frac{\pr^{\abs{\alpha}}}{\pr\ol z^\alpha}=\pr^\alpha_{\ol z}=\pr^{\alpha_1}_{\ol z_1}\cdots\pr^{\alpha_n}_{\ol z_n}$, $\pr_{\ol z_j}=
\frac{\pr}{\pr\ol z_j}=\frac{1}{2}(\frac{\pr}{\pr x_{2j-1}}+i\frac{\pr}{\pr x_{2j}})$, $j=1,\ldots,n$.
For $j, s\in\mathbb Z$, set $\delta_{j,s}=1$ if $j=s$, $\delta_{j,s}=0$ if $j\neq s$. 

Let $M$ be a $C^\infty$ paracompact manifold. 
We let $TM$ and $T^*M$ denote the tangent bundle of $M$ and the cotangent bundle of $M$ respectively.
The complexified tangent bundle of $M$ and the complexified cotangent bundle of $M$ will be denoted by $\Complex TM$ and $\Complex T^*M$ respectively. We write $\langle\,\cdot\,,\cdot\,\rangle$ to denote the pointwise duality between $TM$ and $T^*M$.
We extend $\langle\,\cdot\,,\cdot\,\rangle$ bilinearly to $\Complex TM\times\Complex T^*M$.
Let $E$ be a $C^\infty$ vector bundle over $M$. The fiber of $E$ at $x\in M$ will be denoted by $E_x$.
Let $F$ be another vector bundle over $M$. We write 
$E\boxtimes F$ to denote the vector bundle over $M\times M$ with fiber over $(x, y)\in M\times M$ 
consisting of the linear maps from $E_x$ to $F_y$.  Let $Y\subset M$ be an open set. From now on, the spaces of
smooth sections of $E$ over $Y$ and distribution sections of $E$ over $Y$ will be denoted by $C^\infty(Y, E)$ and $\mathscr D'(Y, E)$ respectively.
Let $\mathscr E'(Y, E)$ be the subspace of $\mathscr D'(Y, E)$ whose elements have compact support in $Y$.
For $m\in\Real$, we let $H^m(Y, E)$ denote the Sobolev space
of order $m$ of sections of $E$ over $Y$. Put
\begin{gather*}
H^m_{\rm loc\,}(Y, E)=\big\{u\in\mathscr D'(Y, E);\, \varphi u\in H^m(Y, E),
      \, \forall\varphi\in C^\infty_0(Y)\big\}\,,\\
      H^m_{\rm comp\,}(Y, E)=H^m_{\rm loc}(Y, E)\cap\mathscr E'(Y, E)\,.
\end{gather*} 

Let $E$ and $F$ be $C^\infty$ vector
bundles over a paracompact $C^\infty$ manifold $M$ equipped with a smooth density of integration. If
$A: C^\infty_0(M,E)\To \mathscr D'(M,F)$
is continuous, we write $K_A(x, y)$ or $A(x, y)$ to denote the distribution kernel of $A$.
The following two statements are equivalent
\begin{enumerate}
\item $A$ is continuous: $\mathscr E'(M,E)\To C^\infty(M,F)$,
\item $K_A\in C^\infty(M\times M,E_y\boxtimes F_x)$.
\end{enumerate}
If $A$ satisfies (a) or (b), we say that $A$ is smoothing. Let
$B: C^\infty_0(M,E)\to \mathscr D'(M,F)$ be a continuous operator. 
We write $A\equiv B$ (on $M$) if $A-B$ is a smoothing operator. We say that $A$ is properly supported if ${\rm Supp\,}K_A\subset M\times M$ is proper. That is, the two projections: $t_x:(x,y)\in{\rm Supp\,}K_A\To x\in M$, $t_y:(x,y)\in{\rm Supp\,}K_A\To y\in M$ are proper (i.e. the inverse images of $t_x$ and $t_y$ of all compact subsets of $M$ are compact). 

Let $H(x,y)\in\mathscr D'(M\times M,E_y\boxtimes F_x)$. We write $H$ to denote the unique continuous operator $C^\infty_0(M,E)\To\mathscr D'(M,F)$ with distribution kernel $H(x,y)$. In this work, we identify $H$ with $H(x,y)$. 

We recall H\"ormander symbol spaces 

\begin{defn}\label{d-gue140329}
Let $M\subset\Real^N$ be an open set, $0\leq\rho\leq1$, $0\leq\delta\leq1$, $m\in\Real$, $N_1\in\mathbb N$. $S^m_{\rho,\delta}(M\times\Real^{N_1})$ is the space of all $a\in C^\infty(M\times\Real^{N_1})$ such that for all compact $K\Subset M$ and all $\alpha\in\mathbb N^N_0$, $\beta\in\mathbb N^{N_1}_0$, there is a constant $C>0$ such that 
\[\abs{\pr^\alpha_x\pr^\beta_\theta a(z,\theta)}\leq C(1+\abs{\theta})^{m-\rho\abs{\beta}+\delta\abs{\alpha}},\ \ (x,\theta)\in K\times\Real^{N_1}.\]
We say that $S^m_{\rho,\delta}$ is the space of symbols of order $m$ type $(\rho,\delta)$. Put 
\[S^{-\infty}(M\times\Real^{N_1}):=\bigcap_{m\in\Real}S^m_{\rho,\delta}(M\times\Real^{N_1}).\]
Let $a_j\in S^{m_j}_{\rho,\delta}(M\times\Real^{N_1})$, $j=0,1,2,\ldots$ with $m_j\To-\infty$, $j\To\infty$. Then there exists $a\in S^{m_0}_{\rho,\delta}(M\times\Real^{N_1})$ unique modulo $S^{-\infty}(M\times\Real^{N_1})$, such that $a-\sum^{k-1}_{j=0}a_j\in S^{m_k}_{\rho,\delta}(M\times\Real^{N_1})$ for $k=0,1,2,\ldots$. 

If $a$ and $a_j$ have the properties above, we write $a\sim\sum^{\infty}_{j=0}a_j$ in $S^{m_0}_{\rho,\delta}(M\times\Real^{N_1})$. 

Let $S^m_{{\rm cl\,}}(M\times\Real^{N_1})$ be the space of all symbols $a(x,\theta)\in S^m_{1,0}(M\times\Real^{N_1})$ with 
\[\mbox{$a(x,\theta)\sim\sum^\infty_{j=0}a_{m-j}(x,\theta)$ in $S^m_{1,0}(M\times\Real^{N_1})$},\]
with $a_k(x,\theta)\in C^\infty(M\times\Real^{N_1})$ positively homogeneous of degree $k$ in $\theta$, that is, $a_k(x,\lambda\theta)=\lambda^ka_k(x,\theta)$, $\lambda\geq1$, $\abs{\theta}\geq1$. 

By using partition of unity, we extend the definitions above to the cases when $M$ is a smooth paracompact manifold and when 
we replace $M\times\Real^{N_1}$ by $T^*M$.
\end{defn}

Let $\Omega\subset X$ be an open set. Let
$a(x, \xi)\in S^k_{\frac{1}{2},\frac{1}{2}}(T^*\Omega)$. We can define
\[A(x, y)=\frac{1}{(2\pi)^{3}}\int\! e^{i<x-y,\xi>}a(x,\xi)d\xi\]
as an oscillatory integral and we can show that
\[A:C^\infty_0(\Omega)\To C^\infty(\Omega)\]
is continuous and has unique continuous extension:
\[A:\mathscr E'(\Omega)\To\mathscr D'(\Omega).\]

\begin{defn} \label{d:ss-pseudo}
Let $k\in\Real$. A pseudodifferential operator of order $k$ type
$(\frac{1}{2},\frac{1}{2})$ is a continuous linear map
$A:C^\infty_0(\Omega)\To\mathscr D'(\Omega)$
such that the distribution kernel of $A$ is
\[A(x, y)=\frac{1}{(2\pi)^3}\int\! e^{i<x-y,\xi>}a(x, \xi)d\xi\]
with $a\in S^k_{\frac{1}{2},\frac{1}{2}}(T^*\Omega)$. We call $a(x, \xi)$ the symbol of $A$. We shall write
$L^k_{\frac{1}{2},\frac{1}{2}}(\Omega)$
to denote the space of pseudodifferential operators of order $k$ type $(\frac{1}{2},\frac{1}{2})$.
\end{defn}

We recall the following classical result of Calderon-Vaillancourt~(see chapter~{\rm XVIII} of H\"{o}rmander~\cite{Hor85}).

\begin{prop} \label{p:he-calderon}
If $A\in L^k_{\frac{1}{2},\frac{1}{2}}(\Omega)$.
Then,
\[A:H^s_{\rm comp}(\Omega)\To H^{s-k}_{\rm loc}(\Omega)\]
is continuous, for all $s\in\Real$. Moreover,
if $A$ is properly supported, then
\[A:H^s_{\rm loc}(\Omega)\To H^{s-k}_{\rm loc}(\Omega)\]
is continuous, for all $s\in\Real$.
\end{prop}

\section{Microlocal analysis for $\Box_b$}\label{s-gue140325II}

We will reduce the analysis of the Paneitz operator to the analysis of Kohn Laplacian. We extend $\ddbar_b$ to $L^2$ space by $\ddbar_b:{\rm Dom\,}\ddbar_b\subset L^2(X)\To L^2_{(0,1)}(X)$, where 
${\rm Dom\,}\ddbar_b:=\set{u\in L^2(X);\, \ddbar_bu\in L^2_{(0,1)}(X)}$. Let 
\[\ol{\pr}^*_b:{\rm Dom\,}\ol{\pr}^*_b\subset L^2_{(0,1)}(X)\To L^2(X)\] 
be the $L^2$ adjoint of $\ddbar_b$. The Gaffney extension of Kohn Laplacian is given by 
\begin{equation}\label{e-gue140325aIII}
\begin{split}
&\Box_b=\ol{\pr}^*_b\ddbar_b:{\rm Dom\,}\Box_b\subset L^2(X)\To L^2(X),\\
&{\rm Dom\,}\Box_b:=\set{u\in L^2(X);\, u\in{\rm Dom\,}\ddbar_b, \ddbar_bu\in{\rm Dom\,}\ol{\pr}^*_b}.
\end{split}
\end{equation}
It is well-known that $\Box_b$ is a positive self-adjoint operator. Moreover, the characteristic manifold of $\Box_b$ is given by 
\begin{equation}\label{e-gue140329}
\Sigma=\set{(x,\xi)\in T^*X;\, \xi=\lambda\theta(x),\ \lambda\neq0}.
\end{equation}

Since $X$ is embeddable, $\Box_b$ has $L^2$ closed range. Let $G:L^2(X)\To{\rm Dom\,}\Box_b$ be the partial inverse and let $S:L^2(X)\To{\rm Ker\,}\Box_b$ be the orthogonal projection (Szeg\"o projection). Then, 
\begin{equation}\label{e-gue140325aIV}
\begin{split}
&\mbox{$\Box_bG+S=I$ on $L^2(X)$},\\
&\mbox{$G\Box_b+S=I$ on ${\rm Dom\,}\Box_b$}.
\end{split}
\end{equation}
In~\cite{Hsiao08}, we proved that $G\in L^{-1}_{\frac{1}{2},\frac{1}{2}}(X)$, $S\in L^{0}_{\frac{1}{2},\frac{1}{2}}(X)$ and we got explicit formulas of the kernels $G(x,y)$ and $S(x,y)$. 

We introduce some notations. 
Let $M$ be an open set in $\Real^N$ and let $f$, $g\in C^\infty(M)$. We write $f\asymp g$
if for every compact set $K\subset M$ there is a constant $c_K>0$ such that
$f\leq c_Kg$ and $g\leq c_Kf$ on $K$. Let $\Omega\subset X$ be an open set with real local coordinates $x=(x_1,x_2,x_3)$. We need 

\begin{defn} \label{d-dhikbmiI}
$a(t,x,\eta)\in C^\infty(\ol\Real_+\times T^*\Omega)$
is quasi-homogeneous of
degree $j$ if $a(t,x,\lambda\eta)=\lambda^ja(\lambda t,x,\eta)$ for all $\lambda>0$.
\end{defn}

We introduce some symbol classes 

\begin{defn} \label{d-d-msmapamo}
Let $\mu>0$. We say that $a(t,x,\eta)\in\Td S^m_{\mu}(\ol\Real_+\times T^*\Omega)$
if $a(t,x,\eta)\in C^\infty(\ol\Real_+\times T^*\Omega)$ and there is a $a(x,\eta)\in S^m_{1,0}(T^*\Omega)$ such that for all indices $\alpha, \beta\in\mathbb N^{3}_0$, $\gamma\in\mathbb N_0$, every compact set $K\Subset\Omega$, there exists a constant $c_{\alpha,\beta,\gamma}>0$ independent of $t$ such that for all $t\in\ol\Real_+$, 
\[\abs{\pr^\gamma_t\pr^\alpha_{x}\pr^\beta_{\eta}(a(t,x,\eta)-a(x,\eta))}\leq
 c_{\alpha,\beta,\gamma}e^{-t\mu\abs{\eta}}(1+\abs{\eta})^{m+\gamma-\abs{\beta}},\; x\in K, \abs{\eta}\geq1.\] 
 


\end{defn}

The following is well-known (see~\cite{Hsiao08})

\begin{thm}\label{t-gue140325II}
With the assumptions and notations above, $G\in L^{-1}_{\frac{1}{2},\frac{1}{2}}(X)$, $S\in L^{0}_{\frac{1}{2},\frac{1}{2}}(X)$, $S(x,y)\equiv\int e^{i\varphi(x,y)t}a(x,y,t)dt$, where $\varphi(x,y)\in C^\infty(X\times X)$ is as in \eqref{e-gue140325VIII} and \[\begin{split}
&a(x, y, t)\in S^{1}_{{\rm cl\,}}(X\times X\times]0, \infty[), \\
&\mbox{$a(x, y, t)\sim\sum^\infty_{j=0}a_j(x, y)t^{1-j}$ in $S^1_{1,0}(X\times X\times]0,\infty[)$},\\
&a_j(x,y)\in C^\infty(X\times X),\ \ j=0,1,\ldots,\\
&a_0(x,x)=\frac{1}{2}\pi^{-n},\ \ \forall x\in X,
\end{split}\]
and on every open local coordinate patch $\Omega\subset X$ with real local coordinates $x=(x_1,x_2,x_3)$, we have  
\begin{equation}\label{e-gue140329f}
G(x,y)\equiv \int\int^{\infty}_0e^{i(\psi(t,x,\eta)-<y,\eta>)}-t\bigr(i\psi'_t(t,x,\eta)a(t,x,\eta)+\frac{\pr a}{\pr t}(t,x,\eta)\bigr)dt d\eta,
\end{equation}
where $a(t,x,\eta)\in\Td S^0_{\mu}(\ol\Real_+\times T^*\Omega)$,  $\psi(t,x,\eta)\in\Td S^1_{\mu}(\ol\Real_+\times T^*\Omega)$ for some $\mu>0$, $\psi(t,x,\eta)$ is quasi-homogeneous of degree $1$, $\psi(0,x,\eta)=<x,\eta>$, ${\rm Im\,}\psi\geq 0$ with equality precisely on $(\set{0}\times T^*\Omega\setminus0)
\bigcup(\Real_+\times\Sigma)$, 
\[\psi(t,x,\eta)=<x,\eta>\text{ on }\Sigma,\; d_{x,\eta}(\psi-<x,\eta>)=0\text{ on }\Sigma,\]
and
\begin{equation} \label{e:c-sj1}
{\rm Im\,}\psi(t, x,\eta)\asymp\Big(\abs{\eta}\frac{t\abs{\eta}}{1+t\abs{\eta}}\Big)\Big({\rm dist\,}
\big((x, \frac{\eta}{\abs{\eta}}),\Sigma\big)\Big)^2,\ \ t\geq0,\ \abs{\eta}\geq1.
\end{equation}
(See Theorem~\ref{t-gue140331} below  for the meaning of the integral \eqref{e-gue140329f}.)
\end{thm}

\begin{proof}
We only sketch the proof. For all the details, we refer the reader to Part I in \cite{Hsiao08}.
We use the heat equation method. We work with some
real local coordinates $x=(x_1,x_2,x_3)$ defined on $\Omega$. 
We consider the problem
\begin{equation} \label{e:i-heat}
\left\{ \begin{array}{ll}
(\pr_t+\Box_b)u(t,x)=0  & \text{ in }\Real_+\times\Omega,  \\
u(0,x)=v(x). \end{array}\right.
\end{equation}
We look for an approximate solution of
\eqref{e:i-heat} of the form $u(t,x)=A(t)v(x)$,
\begin{equation} \label{e:i-fourierheat}
A(t)v(x)=\frac{1}{(2\pi)^{3}}\int\!e^{i(\psi(t,x,\eta)-\langle y,\eta\rangle)}\alpha(t,x,\eta)v(y)dyd\eta
\end{equation}
where formally
\[\alpha(t,x,\eta)\sim\sum^\infty_{j=0}\alpha_j(t,x,\eta),\]
with $\alpha_j(t,x,\eta)$ quasi-homogeneous of degree $-j$.

The full symbol of $\Box_b$ equals $\sum^2_{j=0}p_j(x,\xi)$,
where $p_j(x,\xi)$ is positively homogeneous of order $2-j$ in the sense that
\[p_j(x, \lambda\eta)=\lambda^{2-j}p_j(x, \eta),\ \abs{\eta}\geq1,\ \lambda\geq1.\]
We apply $\pr_t+\Box_b$ formally
inside the integral in \eqref{e:i-fourierheat} and then introduce the asymptotic expansion of
$\Box_b(\alpha e^{i\psi})$. Set $(\pr_t+\Box_b)(\alpha e^{i\psi})\sim 0$ and regroup
the terms according to the degree of quasi-homogeneity. The phase $\psi(t, x, \eta)$ should solve
\begin{equation} \label{e:i-chara}
\left\{ \begin{array}{ll}
\frac{\displaystyle\pr\psi}{\displaystyle\pr t}-ip_0(x,\psi'_x)=O(\abs{{\rm Im\,}\psi}^N), & \forall N\geq 0,   \\
 \psi|_{t=0}=\langle x, \eta\rangle. \end{array}\right.
\end{equation}
This equation can be solved with ${\rm Im\,}\psi(t, x,\eta)\geq0$ and the phase $\psi(t, x, \eta)$ is quasi-homogeneous of
degree $1$. Moreover,
\begin{gather*}
\psi(t,x,\eta)=\langle x,\eta\rangle{\;\rm on\;}\Sigma,\; d_{x,\eta}(\psi-\langle x,\eta\rangle)=0 \text{ on }\Sigma,\\
{\rm Im\,}\psi(t, x,\eta)\asymp\Big(\abs{\eta}\frac{t\abs{\eta}}{1+t\abs{\eta}}\Big)
\Big({\rm dist\,}\big((x,\frac{\eta}{\abs{\eta}}),\Sigma\big)\Big)^2,\ \ \abs{\eta}\geq 1.
\end{gather*}
Furthermore, there exists $\psi(\infty,x,\eta)\in C^\infty(\Omega\times\dot\Real^3)$ with a
uniquely determined Taylor expansion at each point of $\Sigma$ such that for every compact set
$K\subset\Omega\times\dot\Real^{3}$ there is a constant $c_K>0$ such that
\[{\rm Im\,}\psi(\infty,x,\eta)\geq c_K\abs{\eta}\Big({\rm dist\,}\big((x,\frac{\eta}{\abs{\eta}}),\Sigma\big)\Big)^2,
\ \ \abs{\eta}\geq 1.\]
If $\lambda\in C(T^*\Omega\smallsetminus 0)$, $\lambda>0$ is positively homogeneous of degree $1$ and 
$\lambda|_\Sigma<\min\lambda_j$, $\lambda_j>0$, where
$\pm i\lambda_j$ are the non-vanishing eigenvalues of the fundamental matrix of $\Box_b$,
then the solution $\psi(t,x,\eta)$ of \eqref{e:i-chara} can be chosen so that for every
compact set $K\subset\Omega\times\dot\Real^{3}$ and all indices $\alpha$, $\beta$, $\gamma$,
there is a constant $c_{\alpha,\beta,\gamma,K}$ such that
\[ \abs{\pr^\alpha_x\pr^\beta_\eta\pr^\gamma_t(\psi(t,x,\eta)-\psi(\infty,x,\eta))}
\leq c_{\alpha,\beta,\gamma,K}e^{-\lambda(x,\eta)t}\text{ on }\ol\Real_+\times K.\]
We obtain the transport equations
\begin{equation} \label{e:i-heattransport}
\left\{ \begin{array}{l}
 T(t,x,\eta,\pr_t,\pr_x)\alpha_0=O(\abs{{\rm Im\,}\psi}^N), \ \forall N,   \\
 T(t,x,\eta,\pr_t,\pr_x)\alpha_j+l_j(t,x,\eta,\alpha_0,\ldots,\alpha_{j-1})= O(\abs{{\rm Im\,}\psi}^N), \ \forall N,\ \ j\in\mathbb N.
 \end{array}\right.
\end{equation}

It was proved in~\cite{Hsiao08} that \eqref{e:i-heattransport} can be solved. Moreover, there exist positively homogeneous functions of degree $-j$
\[\alpha_j(\infty, x, \eta)\in C^\infty(T^*\Omega),\ \ j=0,1,2,\ldots,\]
such that $\alpha_j(t, x, \eta)$ converges exponentially
fast to $\alpha_j(\infty, x, \eta)$, $t\to\infty$, for all $j\in\N_0$.
Set
\[\Td G=\frac{1}{(2\pi)^{3}}\int\int^{\infty}_0e^{i(\psi(t,x,\eta)-<y,\eta>)}-t\bigr(i\psi'_t(t,x,\eta)\alpha(t,x,\eta)+\frac{\pr \alpha}{\pr t}(t,x,\eta)\bigr)dt d\eta\]
and
\[\Td S=\frac{1}{(2\pi)^{3}}\int e^{i(\psi(\infty,x,\eta)-\langle y,\eta\rangle)}\alpha(\infty,x,\eta)d\eta.\]
We can show that $\Td G$ is a pseudodifferential operator of order $-1$ type $(\frac{1}{2},\frac{1}{2})$,  $\Td S$ is a pseudodifferential operator of order $0$ type $(\frac{1}{2},\frac{1}{2})$ satisfying
\[\Td S+\Box_b\Td G\equiv I,\ \ \Box_b\Td S\equiv0.\]
Moreover, from global theory of complex Fourier integral operators, we can show that 
$\Td S\equiv\int e^{i\varphi(x,y)t}a(x,y,t)dt$. Furthermore, by using some standard argument in functional analysis, we can show that $\Td G\equiv G$, $\Td S\equiv S$. 
\end{proof}

Until further notice, we work in an open local coordinate patch $\Omega\subset X$ with real local coordinates $x=(x_1,x_2,x_3)$. The following is well-known (see Chapter 5 in~\cite{Hsiao08})

\begin{thm}\label{t-gue140331}
With the notations and assumptions used in Theorem~\ref{t-gue140325II}, let $\chi\in C^\infty_0(\Real^3)$ be equal to $1$ near the origin. Put 
\[\begin{split}
&G_\varepsilon(x,y)\\
&=\int\int^{\infty}_0e^{i(\psi(t,x,\eta)-<y,\eta>)}-t\bigr(i\psi'_t(t,x,\eta)a(t,x,\eta)+\frac{\pr a}{\pr t}(t,x,\eta)\bigr)\chi(\varepsilon\eta)dt d\eta,\end{split}\]
where $\psi(t,x,\eta)$, $a(t,x,\eta)$ are as in \eqref{e-gue140329f}. For $u\in C^\infty_0(\Omega)$, we can show that 
\[Gu:=\lim_{\varepsilon\To0}\int G_\varepsilon(x,y)u(y)dy\in C^\infty(\Omega)\]
and 
\[\begin{split}
G:C^\infty_0(\Omega)&\To C^\infty(\Omega),\\
u&\To\lim_{\varepsilon\To0}\int G_\varepsilon(x,y)u(y)dy
\end{split}\]
is continuous. 

Moreover, $G\in L^{-1}_{\frac{1}{2},\frac{1}{2}}(\Omega)$ with symbol 
\[\int^{\infty}_0e^{i(\psi(t,x,\eta)-<x,\eta>)}-t\bigr(i\psi'_t(t,x,\eta)a(t,x,\eta)+\frac{\pr a}{\pr t}(t,x,\eta)\bigr)dt\in S^{-1}_{\frac{1}{2},\frac{1}{2}}(T^*\Omega).\]
\end{thm}

We need the following (see Lemma 5.13 in~\cite{Hsiao08} for a proof)

\begin{lem}\label{l-gue140331}
With the notations and assumptions used in Theorem~\ref{t-gue140325II}, 
for every compact set $K\subset\Omega$ and all $\alpha\in\mathbb N^3_0$, $\beta\in\mathbb N^3_0$,
there exists a constant $c_{\alpha,\beta,K}>0$ such that
\begin{equation} \label{e-gue140331}\begin{split}
&\abs{\pr^\alpha_x\pr^\beta_\eta(e^{i(\psi(t, x, \eta)-<x, \eta>)}t\psi'_t(t, x, \eta))}  \\
&\quad\leq c_{\alpha,\beta,K}(1+\abs{\eta})^
{\frac{\abs{\alpha}-\abs{\beta}}{2}}e^{-t\mu\abs{\eta}}e^{-{\rm Im\,}\psi(t, x, \eta)}
(1+{\rm Im\,}\psi(t, x,\eta))^{1+\frac{\abs{\alpha}+\abs{\beta}}{2}},
\end{split}\end{equation}
where $x\in K$, $t\in\ol\Real_+$, $\abs{\eta}\geq1$ and $\mu>0$ is a constant independent of $\alpha$, $\beta$ and $K$. 
\end{lem}

In this work, we need 

\begin{thm}\label{t-gue140325III}
Let $L\in C^\infty(X,T^{1,0}X\oplus T^{0,1}X)$. Then, $L\circ G\in L^{-\frac{1}{2}}_{\frac{1}{2},\frac{1}{2}}(X)$. 
\end{thm}

\begin{proof}
We work on an open local coordinate patch $\Omega\subset X$ with real local coordinates $x=(x_1,x_2,x_3)$.
Let $l(x,\eta)\in C^\infty(T^*\Omega)$ be the symbol of $L$. Then, $l(x,\lambda\eta)=\lambda l(x,\eta)$, $\lambda>0$. It is well-known (see Chapter 5 in~\cite{Hsiao08}) that 
\[(LG)(x,y)\equiv\int e^{i<x-y,\eta>}\alpha(x,\eta)d\eta,\]
where 
\[\begin{split}
&\alpha(x,\eta)=\alpha_0(x,\eta)+\alpha_1(x,\eta)\in S^0_{\frac{1}{2},\frac{1}{2}}(T^*\Omega),\\
&\alpha_0(x,\eta)=\int e^{i(\psi(t, x, \eta)-<x, \eta>)}(-1)l(x,\psi'_x(t,x,\eta))t\psi'_t(t, x, \eta)a(t,x,\eta)dt,\\
&\alpha_1(x,\eta)\in S^{-1}_{\frac{1}{2},\frac{1}{2}}(T^*\Omega).
\end{split}\]
Here $a(t,x,\eta)\in \Td S^0_{\mu}(\ol\Real_+\times T^*\Omega)$, $\mu>0$.
We only need to prove that $\alpha_0(x,\eta)\in S^{-\frac{1}{2}}_{\frac{1}{2},\frac{1}{2}}(T^*\Omega)$. Fix $\alpha, \beta\in\mathbb N^3_0$. From \eqref{e-gue140331}, \eqref{e:c-sj1} and notice that $l(x,\psi'_x(t,x,\eta))=0$ at $\Sigma$, we can check that 
\begin{equation}\label{e-gue140326}
\begin{split}
&\abs{\pr^\alpha_x\pr^\beta_\eta\alpha_0(x,\eta)}\\
&\leq\sum_{\abs{\alpha'}+\abs{\alpha''}=\abs{\alpha},\abs{\beta'}+\abs{\beta''}=\abs{\beta}}\int\abs{\pr^{\alpha'}_x\pr^{\beta'}_\eta\bigr(e^{i(\psi(t, x, \eta)-<x, \eta>)}t\psi'_t(t, x, \eta)\bigr)}\times\\
&\quad\quad\quad\quad\quad\quad\quad\quad\abs{\pr^{\alpha''}_x\pr^{\beta''}_\eta\bigr(l(x,\psi'_x(t,x,\eta))a(t,x,\eta)\bigr)}dt\\
&\leq C_{\alpha,\beta}\sum_{\abs{\alpha'}+\abs{\alpha''}=\abs{\alpha},\abs{\beta'}+\abs{\beta''}=\abs{\beta}}\int(1+\abs{\eta})^
{\frac{\abs{\alpha'}-\abs{\beta'}}{2}}e^{-t\mu\abs{\eta}}e^{-\frac{1}{2}{\rm Im\,}\psi(t, x, \eta)}\times\\
&\quad\quad\quad\quad\quad\quad\quad\quad(1+\abs{\eta})^{1-\abs{\beta''}}\Big({\rm dist\,}
\big((x, \frac{\eta}{\abs{\eta}}),\Sigma\big)\Big)^{\max\{0,1-\abs{\beta''}\}}dt\\
&\leq\Td C_{\alpha,\beta}\sum_{\abs{\alpha'}+\abs{\alpha''}=\abs{\alpha},\abs{\beta'}+\abs{\beta''}=\abs{\beta}}\int(1+\abs{\eta})^
{\frac{\abs{\alpha'}-\abs{\beta'}}{2}}e^{-c\frac{t\abs{\eta}^2}{1+t\abs{\eta}}\Big({\rm dist\,}
\big((x, \frac{\eta}{\abs{\eta}}),\Sigma\big)\Big)^2}\times\\
&\quad\quad\quad\quad\quad\quad\quad\quad e^{-t\mu\abs{\eta}}(1+\abs{\eta})^{1-\abs{\beta''}}\Big({\rm dist\,}
\big((x, \frac{\eta}{\abs{\eta}}),\Sigma\big)\Big)^{\max\{0,1-\abs{\beta''}\}}dt,
\end{split}
\end{equation}
where $c>0$, $\mu>0$, $C_{\alpha,\beta}>0$ and $\Td C_{\alpha,\beta}>0$ are constants. 

When $\abs{\beta''}=0$, we have 
\begin{equation}\label{e-gue140326I}
\begin{split}
&\int(1+\abs{\eta})^
{\frac{\abs{\alpha'}-\abs{\beta'}}{2}}e^{-c\frac{t\abs{\eta}^2}{1+t\abs{\eta}}\Big({\rm dist\,}
\big((x, \frac{\eta}{\abs{\eta}}),\Sigma\big)\Big)^2}\times\\
&\quad\quad\quad\quad\quad\quad\quad\quad e^{-t\mu\abs{\eta}}(1+\abs{\eta})^{1-\abs{\beta''}}\Big({\rm dist\,}
\big((x, \frac{\eta}{\abs{\eta}}),\Sigma\big)\Big)dt\\
&\leq \Td c\int(1+\abs{\eta})^{\frac{\abs{\alpha'}-\abs{\beta'}}{2}}\frac{1}{\sqrt{t}}(1+\abs{\eta})^{-\abs{\beta''}}e^{-\frac{1}{2}t\mu\abs{\eta}}dt\\
&\leq\Td c_1\int^{\frac{1}{1+\abs{\eta}}}_0(1+\abs{\eta})^{\frac{\abs{\alpha'}-\abs{\beta'}}{2}}\frac{1}{\sqrt{t}}(1+\abs{\eta})^{-\abs{\beta''}}e^{-\frac{1}{2}t\mu\abs{\eta}}dt\\
&+\Td c_2\int^{\infty}_{\frac{1}{1+\abs{\eta}}}(1+\abs{\eta})^{\frac{\abs{\alpha'}-\abs{\beta'}}{2}}\frac{1}{\sqrt{t}}(1+\abs{\eta})^{-\abs{\beta''}}e^{-\frac{1}{2}t\mu\abs{\eta}}dt\\
&\leq\Td c_3(1+\abs{\eta})^{-\frac{1}{2}-\abs{\beta''}+\frac{\abs{\alpha'}-\abs{\beta'}}{2}},
\end{split}
\end{equation}
where $\abs{\eta}\geq1$, $\Td c_1>0$, $\Td c_2>0$ and $\Td c_3>0$ are constants. 

When $\abs{\beta''}\geq1$, we have 
\begin{equation}\label{e-gue140326II}
\begin{split}
&\int(1+\abs{\eta})^
{\frac{\abs{\alpha'}-\abs{\beta'}}{2}}e^{-c\frac{t\abs{\eta}^2}{1+t\abs{\eta}}\Big({\rm dist\,}
\big((x, \frac{\eta}{\abs{\eta}}),\Sigma\big)\Big)^2}\times\\
&\quad\quad\quad\quad\quad\quad\quad\quad e^{-t\mu\abs{\eta}}(1+\abs{\eta})^{1-\abs{\beta''}}dt\\
&\leq \hat c\int(1+\abs{\eta})^{\frac{\abs{\alpha'}-\abs{\beta'}}{2}}(1+\abs{\eta})^{1-\abs{\beta''}}e^{-t\mu\abs{\eta}}dt\\
&\leq\hat c_1(1+\abs{\eta})^{-\abs{\beta''}+\frac{\abs{\alpha'}-\abs{\beta'}}{2}}\\
&\leq\hat c_2(1+\abs{\eta})^{-\frac{1}{2}+\frac{\abs{\alpha'}-\abs{\beta'}-\abs{\beta''}}{2}},
\end{split}
\end{equation}
where $\abs{\eta}\geq1$, $\hat c_1>0$, $\hat c_2>0$ are constants. 

From \eqref{e-gue140326}, \eqref{e-gue140326I} and \eqref{e-gue140326II}, we conclude that $\alpha_0(x,\eta)\in S^{-\frac{1}{2}}_{\frac{1}{2},\frac{1}{2}}(T^*\Omega)$. The theorem follows. 
\end{proof}

\section{Microlocal Hodge decomposition theorems for ${\rm P\,}$ and the proof of Theorem~\ref{t-gue140325}}\label{s-gue140326}

By using Theorem~\ref{t-gue140325II} and Theorem~\ref{t-gue140325III}, we will establish microlocal Hodge decomposition theorems for ${\rm P\,}$ in this section. Let $G\in L^{-1}_{\frac{1}{2},\frac{1}{2}}(X)$, $S\in L^0_{\frac{1}{2},\frac{1}{2}}(X)$ be as in Theorem~\ref{t-gue140325II}. From \eqref{e-gue140325} and \eqref{e-gue140325aIV}, we have 
\begin{equation}\label{e-gue140326III}
\begin{split}
{\rm P\,}\ol GG&=(\Box_b\ol\Box_b+L_1\circ L_2+L_3)\ol GG\\
&=\Box_b(I-\ol S)G+L_1\circ L_2\ol GG+L_3\ol GG\\
&=I-S-\Box_b\ol SG+L_1\circ L_2\ol GG+L_3\ol GG\\
&=I-S-\ol S\Box_bG+\ol S\Box_bG-\Box_b\ol SG+L_1\circ L_2\circ\ol GG+L_3\ol GG\\
&=I-S-\ol S(I-S)+[\ol S,\Box_b]G+L_1\circ L_2\ol GG+L_3\ol GG\\
&=I-S-\ol S+\ol SS+[\ol S,\Box_b]G+L_1\circ L_2\ol GG+L_3\ol GG.
\end{split}
\end{equation}
We need 

\begin{lem}\label{l-gue140326}
We have 
\begin{equation}\label{e-gue140326IV}
\begin{split}
&[\ol S,\Box_b]G+L_1\circ L_2\ol GG+L_3\ol GG\\
&\mbox{$:H^s(X)\To H^{s+\frac{1}{2}}(X)$ is continuous, for every $s\in\mathbb Z$}.
\end{split}
\end{equation}
\end{lem}

\begin{proof}
From Theorem~\ref{t-gue140325III}, we see that $L_1\circ L_2\ol G\in L^{\frac{1}{2}}_{\frac{1}{2},\frac{1}{2}}(X)$. Thus, 
\begin{equation}\label{e-gue140326V}
\mbox{$L_1\circ L_2\ol GG+L_3\ol GG:H^s(X)\To H^{s+\frac{1}{2}}(X)$ is continuous, $\forall s\in\mathbb Z$}.
\end{equation}
Since $\ol\Box_b\ol S=\ol S\,\ol\Box_b=0$, we have 
\begin{equation}\label{e-gue140326VII}
[\ol S,\Box_b]=[\ol S,\Box_b-\ol\Box_b].
\end{equation}
Since the principal symbol of $\Box_b$ is real, $\Box_b-\ol\Box_b$ is a first order partial differential operator. From this observation and note that $\ol S\in L^0_{\frac{1}{2},\frac{1}{2}}(X)$, it is not difficult to see that $[\ol S,\Box_b-\ol\Box_b]\in L^{\frac{1}{2}}_{\frac{1}{2},\frac{1}{2}}(X)$. From this and \eqref{e-gue140326VII}, we conclude that 
\begin{equation}\label{e-gue140326VIII}
\mbox{$[\ol S,\Box_b]G:H^s(X)\To H^{s+\frac{1}{2}}(X)$ is continuous, $\forall s\in\mathbb Z$}.
\end{equation}
From \eqref{e-gue140326VIII} and \eqref{e-gue140326V}, \eqref{e-gue140326IV} follows. 
\end{proof}

We also need 

\begin{lem}\label{l-gue140326I}
We have $\ol SS\equiv0$ on $X$, $S\ol S\equiv0$ on $X$. 
\end{lem}

\begin{proof}
We first notice that $\ol S\circ S$ is smoothing away $x=y$. 
We have 
\begin{equation}\label{e-gue140216f}
\begin{split}
\ol S\circ S(x,y)&\equiv\int_{\sigma>0,t>0}e^{-i\ol\varphi(x,w)\sigma+i\varphi(w,y)t}\ol a(x,w,\sigma)a(w,y,t)d\sigma dv_X(w)dt\\
&\equiv\int_{s>0,t>0}e^{it(-\ol\varphi(x,w)s+\varphi(w,y))}t\ol a(x,w,st)a(w,y,t)ds dv_X(w)dt,
\end{split}\end{equation}
where $dv_X=\theta\wedge d\theta$ is the volume form. 
Take $\chi\in C^\infty_0(\Real,[0,1])$ with $\chi=1$ on $[-\frac{1}{2},\frac{1}{2}]$, $\chi=0$ on $]-\infty,-1]\bigcup[1,\infty[$. From \eqref{e-gue140216f}, we have 
\begin{equation}\label{e-gue140216fI}
\begin{split}
&\ol S\circ S(x,y)\equiv I_\varepsilon+II_\varepsilon,\\
&I_\varepsilon=\int_{s>0,t>0}e^{it(-\ol\varphi(x,w)s+\varphi(w,y))}\chi(\frac{\abs{x-w}^2}{\varepsilon})t\ol a(x,w,st)a(w,y,t)ds dv_X(w)dt,\\
&II_\varepsilon=\int_{s>0,t>0}e^{it(-\ol\varphi(x,w)s+\varphi(w,y))}(1-\chi(\frac{\abs{x-w}^2}{\varepsilon}))\\
&\quad\quad\quad\times t\ol a(x,w,st)a(w,y,t)ds dv_X(w)dt,
\end{split}\end{equation}
where $\varepsilon>0$ is a small constant. Since $\varphi(x,w)=0$ if and only if $x=w$, we can integrate by parts with respect to $s$ and conclude that $II_\varepsilon$ is smoothing. Since $\ol S\circ S$ is smoothing away $x=y$, we may assume that $\abs{x-y}<\varepsilon$.
Since $d_w(-\ol\varphi(x,w)s+\varphi(w,y))|_{x=y=w}=-\omega_0(x)(s+1)\neq0$, if $\varepsilon>0$ is small, we can integrate by parts with respect to $w$ and conclude that $I_\varepsilon$ is smoothing. We get $\ol S\circ S\equiv0$ on $X$. Similarly, we can repeat the procedure above and conclude that $S\circ\ol S\equiv0$ on $X$. The lemma follows. 
\end{proof}

Put 
\begin{equation}\label{e-gue140326aI}
R_0=\ol SS+[\ol S,\Box_b]G+L_1\circ L_2\ol GG+L_3\ol GG.
\end{equation}
From Lemma~\ref{l-gue140326} and Lemma~\ref{l-gue140326I}, we see that 
\begin{equation}\label{e-gue140326aII}
\mbox{$R_0:H^s(X)\To H^{s+\frac{1}{2}}(X)$ is continuous, $\forall s\in\mathbb Z$.}
\end{equation}
We can prove 

\begin{thm}\label{t-gue140328}
With the assumptions and notations above, for every $m\in\mathbb N_0$, there are continuous operators 
\[R_m, A_m:C^\infty_0(X)\To\mathscr D'(X)\]
such that 
\begin{equation}\label{e-gue140328}
\begin{split}
&{\rm P\,}A_m+S+\ol S=I+R_m,\\
&\mbox{$A_m:H^s(X)\To H^{s+2}(X)$ is continuous, $\forall s\in\mathbb Z$},\\
&\mbox{$R_m:H^s(X)\To H^{s+\frac{m+1}{2}}(X)$ is continuous, $\forall s\in\mathbb Z$}.
\end{split}
\end{equation}
\end{thm}

\begin{proof}
From \eqref{e-gue140326aI} and \eqref{e-gue140326III}, we have 
\begin{equation}\label{e-gue140328I}
{\rm P\,}\ol GG+S+\ol S=I+R_0.
\end{equation}
Since $(S+\ol S){\rm P\,}=0$, from \eqref{e-gue140328I}, we have 
\begin{equation}\label{e-gue140328II}
(S+\ol S)^2=S+\ol S+(S+\ol S)R_0.
\end{equation}
From Lemma~\ref{l-gue140326I}, we have 
\begin{equation}\label{e-gue140328III}
(S+\ol S)^2=S^2+S\ol S+\ol SS+\ol S^2\equiv S+\ol S.
\end{equation}
From \eqref{e-gue140328II} and \eqref{e-gue140328III}, we conclude that 
\begin{equation}\label{e-gue140328IV}
\mbox{$(S+\ol S)R_0\equiv0$ on $X$}.
\end{equation}
Fix $m\in\mathbb N_0$. From \eqref{e-gue140328I}, we have 
\begin{equation}\label{e-gue140328V}
\begin{split}
&{\rm P\,}\ol GG\bigr(I-R_0+R^2_0+\cdots+(-R_0)^m\bigr)\\
&\quad+(S+\ol S)\bigr(I-R_0+R^2_0+\cdots+(-R_0)^m\bigr)\\
&=\bigr(I+R_0\bigr)\bigr(I-R_0+R^2_0+\cdots+(-R_0)^m\bigr)=I+R_0(-R_0)^{m}.
\end{split}
\end{equation}
From \eqref{e-gue140328IV}, we have 
\begin{equation}\label{e-gue140328VI}
\bigr(S+\ol S\bigr)\bigr(I-R_0+R^2_0+\cdots+(-R_0)^m\bigr)=S+\ol S-F,\ \ \mbox{$F$ is smoothing}. 
\end{equation}
Put $A_m=\ol GG\bigr(I-R_0+R^2_0+\cdots+(-R_0)^m\bigr)$, $R_m=R_0(-R_0)^m+F$. From \eqref{e-gue140328V}, \eqref{e-gue140328VI} and \eqref{e-gue140326aII}, we obtain 
\[
\begin{split}
&{\rm P\,}A_m+S+\ol S=I+R_m,\\
&\mbox{$A_m:H^s(X)\To H^{s+2}(X)$ is continuous, $\forall s\in\mathbb Z$},\\
&\mbox{$R_m:H^s(X)\To H^{s+\frac{m+1}{2}}(X)$ is continuous, $\forall s\in\mathbb Z$}.
\end{split}
\]
The theorem follows. 
\end{proof}

\begin{lem}\label{l-gue140430}
Let $u\in{\rm Dom\,}{\rm P\,}$. Then, $u=u^0+(S+\ol S)u$, for some $u^0\in H^2(X)\bigcap{\rm Dom\,}{\rm P\,}$. 
\end{lem}

\begin{proof}
Fix $m\geq 3$, $m\in\mathbb N$. Let $A_m$, $R_m$ be as in \eqref{e-gue140328} and let $A^*_m$ and $R^*_m$ be the adjoints of $A_m$ and $R_m$ with respect to $(\,\cdot\,|\,\cdot\,)$ respectively. Then, 
\begin{equation}\label{e-gue140430}
A^*_m{\rm P\,}+S+\ol S=I+R^*_m.
\end{equation}
Let $u\in{\rm Dom\,}{\rm P\,}$. Then, ${\rm P\,}u=v\in L^2(X)$. From \eqref{e-gue140430}, it is easy to see that 
\[u=A^*_mv-R^*_mu+(S+\ol S)u.\]
Since $A^*_mv-R^*_mu\in H^2(X)$ and $(S+\ol S)u\in{\rm Ker\,}{\rm P\,}\subset{\rm Dom\,}{\rm P\,}$, the lemma follows. 
\end{proof}

\begin{lem}\label{l-gue140430I}
Let $u\in{\rm Dom\,}{\rm P\,}$. Then, 
\[(\,(S+\ol S)u\,|\,{\rm P\,}g\,)=0,\ \ \forall g\in{\rm Dom\,}{\rm P\,}.\]
\end{lem}

\begin{proof}
Let $u, g\in{\rm Dom\,}{\rm P\,}$. Take $u_j, g_j\in C^\infty(X)$, $j=1,2,\ldots$, $u_j\To u\in L^2(X)$ as $j\To\infty$ and $g_j\To g\in L^2(X)$ as $j\To\infty$. Then, $(S+\ol S)u_j\To(S+\ol S)u$ in $L^2(X)$ as $j\To\infty$ and ${\rm P\,}g_j\To{\rm P\,}g$ in $H^{-4}(X)$ as $j\To\infty$. Thus, 
\begin{equation}\label{e-gue140430I}
(\,(S+\ol S)u\,|\,{\rm P\,}g\,)=\lim_{j\To\infty}(\,(S+\ol S)u_j\,|\,{\rm P\,}g\,)=\lim_{j\To\infty}\lim_{k\To\infty}(\,(S+\ol S)u_j\,|\,{\rm P\,}g_k\,).
\end{equation}
For any $j, k$, $(\,(S+\ol S)u_j\,|\,{\rm P\,}g_k\,)=(\,{\rm P\,}(S+\ol S)u_j\,|\,g_k\,)=0$. From this observation and \eqref{e-gue140430I}, the lemma follows. 
\end{proof}

Now, we can prove 

\begin{thm}\label{t-gue140430}
The operator ${\rm P\,}:{\rm Dom\,}{\rm P\,}\subset L^2(X)\To L^2(X)$ is self-adjoint. 
\end{thm}

\begin{proof}
Let $u, v\in{\rm Dom\,}{\rm P\,}$. From Lemma~\ref{l-gue140430}, we have 
\[\begin{split}
&u=u^0+(S+\ol S)u,\ \ u^0\in H^2(X)\bigcap{\rm Dom\,}{\rm P\,},\\
&v=v^0+(S+\ol S)v,\ \ v^0\in H^2(X)\bigcap{\rm Dom\,}{\rm P\,}.
\end{split}\]
From Lemma~\ref{l-gue140430I}, we see that 
\begin{equation}\label{e-gue140430II}
(\,u\,|\,{\rm P\,}v\,)=(\,u^0\,|\,{\rm P\,}v^0\,),\ \ (\,{\rm P\,}u\,|\,v\,)=(\,{\rm P\,}u^0\,|\,v^0\,).
\end{equation}
Let $g_j, f_j\in C^\infty(X)$, $j=1,2,\ldots$, $g_j\To u^0\in H^2(X)$ as $j\To\infty$ and $f_j\To v^0\in H^2(X)$ as $j\To\infty$. We have 
\begin{equation}\label{e-gue140430III}
\begin{split}
(\,g_j\,|\,{\rm P\,}f_j\,)&=(\,g_j-u^0\,|\,{\rm P\,}f_j\,)+(\,u^0\,|\,{\rm P\,}f_j\,)\\
&=(\,g_j-u^0\,|\,{\rm P\,}f_j\,)+(\,u^0\,|\,{\rm P\,}v^0\,)+(\,u^0\,|\,{\rm P\,}(f_j-v^0)\,).
\end{split}
\end{equation}
Now, 
\begin{equation}\label{e-gue140430IV}
\begin{split}
\abs{(\,g_j-u^0\,|\,{\rm P\,}f_j\,)}&\leq C_0\norm{g_j-u^0}_2\norm{{\rm P\,}f_j}_{-2}\\
&\mbox{$\leq  C_1\norm{g_j-u^0}_2\norm{f_j}_{2}\To0$ as $j\To\infty$}
\end{split}
\end{equation}
and 
\begin{equation}\label{e-gue140430V}
\begin{split}
\abs{(\,u^0\,|\,{\rm P\,}(f_j-v^0)\,)}&\leq C_2\norm{u^0}_2\norm{{\rm P\,}(f_j-v^0}_{-2}\\
&\mbox{$\leq  C_3\norm{u^0}_2\norm{f_j-v^0}_{2}\To0$ as $j\To\infty$},
\end{split}
\end{equation}
where $C_0>0$, $C_1>0$, $C_2>0$, $C_3>0$ are constants and $\norm{\cdot}_s$ denotes the standard Sobolev norm of order $s$ on $X$. From \eqref{e-gue140430III}, \eqref{e-gue140430IV} and \eqref{e-gue140430V}, we obtain 
\begin{equation}\label{e-gue140430VI}
(\,u^0\,|\,{\rm P\,}v^0\,)=\lim_{j\To\infty}(\,g_j\,|\,{\rm P\,}f_j\,).
\end{equation}
For each $j$, it is clearly that $(\,g_j\,|\,{\rm P\,}f_j\,)=(\,{\rm P\,}g_j\,|\,f_j\,)$. We can repeat the procedure above and conclude that 
\[\lim_{j\To\infty}(\,{\rm P\,}g_j\,|\,f_j\,)=(\,{\rm P\,}u^0\,|\,v^0\,).\]
From this observation, \eqref{e-gue140430VI} and \eqref{e-gue140430II}, we conclude that 
\begin{equation}\label{e-gue140430VII}
(\,{\rm P\,}u\,|\,v\,)=(\,u\,|\,{\rm P\,}v\,),\ \ \forall u, v\in{\rm Dom\,}{\rm P\,}.
\end{equation}
Let ${\rm P^*\,}:{\rm Dom\,}{\rm P^*\,}\subset L^2(X)\To L^2(X)$ be the Hilbert space adjoint of ${\rm P\,}$. From \eqref{e-gue140430VII}, we deduce that ${\rm Dom\,}{\rm P\,}\subset{\rm Dom\,}{\rm P^*\,}$ and ${\rm P\,}u={\rm P^*\,}u$, $\forall u\in{\rm Dom\,}{\rm P^*\,}$. 

Let $v\in{\rm Dom\,}{\rm P^*\,}$. By definition, there is a $f\in L^2(X)$ such that
\[(\,v\,|\,{\rm P\,}g\,)=(\,f\,|\,g\,),\ \ \forall g\in{\rm Dom\,}{\rm P\,}.\]
Since $C^\infty(X)\subset{\rm Dom\,}{\rm P\,}$, ${\rm P\,}v=f$ in the sense of distribution. Since $f\in L^2(X)$, $v\in{\rm Dom\,}{\rm P\,}$ and ${\rm P\,}v={\rm P^*\,}v=f$. The theorem follows. 
\end{proof}

\begin{thm}\label{t-gue140328I}
The operator ${\rm P\,}:{\rm Dom\,}{\rm P\,}\subset L^2(X)\To L^2(X)$ has closed range.
\end{thm}

\begin{proof}
Fix $m\in\mathbb N_0$, let $A_m, R_m$ be as in \eqref{e-gue140328} and let $A^*_m$ and $R^*_m$ be the adjoints of $A_m$ and $R_m$ with respect to $(\,\cdot\,|\,\cdot\,)$ respectively. Then, 
\begin{equation}\label{e-gue140328VII}
A^*_m{\rm P\,}+S+\ol S=I+R^*_m. 
\end{equation}
Now, we claim that there is a constant $C>0$ such that 
\begin{equation}\label{e-gue140328VIII}
\norm{{\rm P\,}u}\geq C\norm{u},\ \ \forall u\in{\rm Dom\,}{\rm P\,}\bigcap({\rm Ker\,}{\rm P\,})^{\perp}.
\end{equation}
If the claim is not true, then we can find $u_j\in{\rm Dom\,}{\rm P\,}\bigcap({\rm Ker\,}{\rm P\,})^{\perp}$ with $\norm{u_j}=1$, $j=1,2,\ldots$, such that 
\begin{equation}\label{e-gue140328a}
\norm{{\rm P\,}u_j}\leq\frac{1}{j}\norm{u_j},\ \ j=1,2,\ldots.
\end{equation}
From \eqref{e-gue140328VII}, we have 
\begin{equation}\label{e-gue140328aI}
u_j=A^*_m{\rm P\,}u_j+(S+\ol S)u_j-R_mu_j,\ \ j=1,2,\ldots. 
\end{equation}
Since $u_j\in({\rm Ker\,}{\rm P\,})^{\perp}$, $j=1,2,\ldots$, and ${\rm P\,}(S+\ol S)=0$, we have 
\begin{equation}\label{e-gue140328aII}
(S+\ol S)u_j=0,\ \ j=1,2,\ldots. 
\end{equation}
From \eqref{e-gue140328aI} and \eqref{e-gue140328aII}, we get 
\begin{equation}\label{e-gue140328aIII}
u_j=A^*_m{\rm P\,}u_j-R_mu_j,\ \ j=1,2,\ldots.
\end{equation}
From \eqref{e-gue140328aIII} and Rellich's theorem, we can find subsequence $\{u_{j_s}\}^{\infty}_{s=1}$, $1\leq j_1<j_2<\cdots$, $u_{j_s}\To u$ in $L^2(X)$. From \eqref{e-gue140328a}, we see that ${\rm P\,}u=0$. Hence, $u\in{\rm Ker\,}{\rm P\,}$. Since $u_j\in({\rm Ker\,}{\rm P\,})^{\perp}$, $j=1,2,\ldots$, we get a contradiction. The claim \eqref{e-gue140328VIII} follows. From \eqref{e-gue140328VIII}, the theorem follows. 
\end{proof}

In view of Theorem~\ref{t-gue140430} and Theorem~\ref{t-gue140328I}, we know that ${\rm P\,}$ is self-adjoint and ${\rm P\,}$ has closed range. Let $N:L^2(X)\To{\rm Dom\,}{\rm P\,}$ be the partial inverse and let $\Pi:L^2(X)\To{\rm Ker\,}{\rm P\,}$ be the orthogonal projection. We can prove 

\begin{thm}\label{t-gue140330}
With the notations and assumptions above, we have 
\begin{equation}\label{e-gue140330}
\begin{split}
&\mbox{$\Pi: H^s(X)\To H^s(X)$ is continuous, $\forall s\in\mathbb Z$},\\
&\mbox{$N: H^s(X)\To H^{s+2}(X)$ is continuous, $\forall s\in\mathbb Z$},\\
&\Pi\equiv S+\ol S.
\end{split}
\end{equation}
\end{thm}

\begin{proof}
Fix $m\in\mathbb N_0$. Let $A_m, R_m$ be as in Theorem~\ref{t-gue140328}. Then, 
\[{\rm P\,}A_m+S+\ol S=I+R_m.\]
Thus, 
\begin{equation}\label{e-gue140330I}
\Pi+\Pi R_m=\Pi({\rm P\,}A_m+S+\ol S)=\Pi(S+\ol S)=S+\ol S.
\end{equation}
From \eqref{e-gue140328} and \eqref{e-gue140330I}, we have 
\begin{equation}\label{e-gue140330II}
\mbox{$\Pi-(S+\ol S):H^{-\frac{m+1}{2}}(X)\To L^2(X)$ is continuous}.
\end{equation}
By taking adjoint in \eqref{e-gue140330II}, we get 
\begin{equation}\label{e-gue140330III}
\mbox{$\Pi-(S+\ol S):L^2(X)\To H^{\frac{m+1}{2}}(X)$ is continuous}.
\end{equation}
From \eqref{e-gue140330II} and \eqref{e-gue140330III}, we have 
\begin{equation}\label{e-gue140330IV}
\mbox{$\bigr(\Pi-(S+\ol S)\bigr)^2:H^{-\frac{m+1}{2}}(X)\To H^{\frac{m+1}{2}}(X)$ is continuous}.
\end{equation}
Now, 
\begin{equation}\label{e-gue140330V}
\begin{split}
\bigr(\Pi-(S+\ol S)\bigr)^2&=\Pi-\Pi(S+\ol S)-(S+\ol S)\Pi+(S+\ol S)^2\\
&=\Pi-(S+\ol S)-(S+\ol S)+S+\ol S+S\ol S+\ol SS\\
&\equiv\Pi-(S+\ol S)\ \ (\mbox{here we used Lemma~\ref{l-gue140326I}}).
\end{split}
\end{equation}
From \eqref{e-gue140330IV} and \eqref{e-gue140330V}, we conclude that 
\[\mbox{$\Pi-(S+\ol S):H^{-\frac{m+1}{2}}(X)\To H^{\frac{m+1}{2}}(X)$ is continuous.}\]
Since $m$ is arbitrary, we get 
\begin{equation}\label{e-gue140330VI}
\Pi\equiv S+\ol S. 
\end{equation}

Now, 
\begin{equation}\label{e-gue140330VII}
N({\rm P\,}A_m+S+\ol S)=N(I+R_m).
\end{equation}
Note that $N{\rm P\,}=I-\Pi$, $N\Pi=0$. From this observation, we have 
\begin{equation}\label{e-gue140330VIII}
N({\rm P\,}A_m+S+\ol S)=(I-\Pi)A_m+NF,
\end{equation}
where $F\equiv0$ (here we used \eqref{e-gue140330VI}). From \eqref{e-gue140330VIII} and \eqref{e-gue140330VII}, we have 
\begin{equation}\label{e-gue140330a}
N-A_m=-\Pi A_m+NF-NR_m.
\end{equation}
From \eqref{e-gue140330VI} and \eqref{e-gue140330a}, we have 
\begin{equation}\label{e-gue140330aI}
\begin{split}
N-A^*_m&=-A^*_m\Pi+F^*N-R^*_mN\\
&=-A^*_m\Pi+F^*(-\Pi A_m+NF-NR_m+A_m)\\
&\quad-R^*_m(-\Pi A_m+NF-NR_m+A_m)\\
&\mbox{$:H^s(X)\To H^{s+2}(X)$ is continuous, $\forall-\frac{m+1}{2}\leq s\leq\frac{m-3}{2}$, $s\in\mathbb Z$,}
\end{split}
\end{equation}
where $A^*_m$, $F^*$, $R^*_m$ are adjoints of $A_m$, $F$, $R_m$ respectively. Note that 
\[\mbox{$A^*_m:H^s(X)\To H^{s+2}(X)$ is continuous, $\forall s\in\mathbb Z$}.\]
From this observation, \eqref{e-gue140330aI} and note that $m$ is arbitrary, we conclude that 
\[\mbox{$N:H^s(X)\To H^{s+2}(X)$ is continuous, $\forall s\in\mathbb Z$}.\]
The theorem follows. 
\end{proof}

Let $\tau$ and $\tau_0$ be as in \eqref{e-gue140325IV}. Now, we can prove 

\begin{thm}\label{t-gue140330I}
We have $\tau\equiv\Pi$ on $X$, $\tau_0\equiv\Pi$ on $X$. 
\end{thm}

\begin{proof}
Since $\hat{\mathcal{P}}\subset{\rm Ker\,}{\rm P\,}$, we have $\Pi\tau=\tau$. From this observation and \eqref{e-gue140330}, we get 
\begin{equation}\label{e-gue140330aII}
(S+\ol S)\tau-\tau=F\tau,
\end{equation}
where $F$ is a smoothing operator. It is clearly that $(S+\ol S)\tau=\tau(S+\ol S)=S+\ol S$. From this observation and \eqref{e-gue140330aII}, we get $S+\ol S-\tau=F\tau$ and hence $S+\ol S-\tau=\tau F^*$, where $F^*$ is the adjoint of $F$. Thus, 
\begin{equation}\label{e-gue140330aIII}
(S+\ol S-\tau)(S+\ol S-\tau)=F\tau^2 F^*\equiv0.
\end{equation}
Now, 
\begin{equation}\label{e-gue140330b}
\begin{split}
(S+\ol S-\tau)^2&=(S+\ol S)^2-(S+\ol S)\tau-\tau(S+\ol S)+\tau^2\\
&=S+S\ol S+\ol SS+\ol S-S-\ol S-S-\ol S+\tau\\
&\equiv\tau-(S+\ol S)\ \ (\mbox{here we used Lemma~\ref{l-gue140326I}}).
\end{split}
\end{equation}
From \eqref{e-gue140330b}, \eqref{e-gue140330aIII} and \eqref{e-gue140330}, we get $\tau\equiv\Pi$. 

Similarly, we can repeat the procedure above and conclude that $\tau_0\equiv\Pi$. The theorem follows. 
\end{proof}

From Theorem~\ref{t-gue140430}, Theorem~\ref{t-gue140328I}, Theorem~\ref{t-gue140330}, Theorem~\ref{t-gue140330I} and Theorem~\ref{t-gue140325II}, we get Theorem~\ref{t-gue140325}. 

\begin{cor}\label{c-gue140330}
We have 
\[\hat{\mathcal{P}}^\perp\bigcap{\rm Ker\,}{\rm P\,}\subset C^\infty(X),\ \ \hat{\mathcal{P}_0}^\perp\bigcap{\rm Ker\,}{\rm P\,}\subset C^\infty(X),\ \ \hat{\mathcal{P}_0}^\perp\bigcap\hat{\mathcal{P}}\subset C^\infty(X)\] 
and $\hat{\mathcal{P}}^\perp\bigcap{\rm Ker\,}{\rm P\,}$, $\hat{\mathcal{P}_0}^\perp\bigcap{\rm Ker\,}{\rm P\,}$, $\hat{\mathcal{P}_0}^\perp\bigcap\hat{\mathcal{P}}$ are all finite dimensional.
\end{cor}

\begin{proof}
If $\hat{\mathcal{P}}^\perp\bigcap{\rm Ker\,}{\rm P\,}$ is infinite dimensional, then we can find 
\[f_j\in\hat{\mathcal{P}}^\perp\bigcap{\rm Ker\,}{\rm P\,},\ \ j=1,2,\ldots,\] 
such that $(\,f_j\,|\,f_k\,)=\delta_{j,k}$, $j,k=1,2,\ldots$. Since $f_j\in{\rm Ker\,}{\rm P\,}$, $j=1,2,\ldots$, $f_j=\Pi f_j$, $j=1,2,\ldots$. From Theorem~\ref{t-gue140330I}, we have 
\begin{equation}\label{e-gue140330bI}
f_j=\tau f_j+Ff_j,\ \ j=1,2,3,\ldots,
\end{equation}
where $F$ is a smoothing operator. Since $f_j\in\hat{\mathcal{P}}^\perp$, $j=1,2,\ldots$, $\tau f_j=0$, $j=1,2,\ldots$. From this observation and \eqref{e-gue140330bI}, we get 
\begin{equation}\label{e-gue140330bII}
f_j=Ff_j,\ \ j=1,2,3,\ldots.
\end{equation}
From \eqref{e-gue140330bII} and Rellich's theorem, we can find subsequence $\set{f_{j_s}}^{\infty}_{s=1}$, $1\leq j_1<j_2<\cdots$, $f_{j_s}\To f$ in $L^2(X)$. Since $(\,f_j\,|\,f_k\,)=\delta_{j,k}$, $j,k=1,2,\ldots$, we get a contradiction. Thus, $\hat{\mathcal{P}}^\perp\bigcap{\rm Ker\,}{\rm P\,}$ is finite dimensional. Let $\{f_1,f_2,\ldots,f_d\}$ be an orthonormal frame of  $\hat{\mathcal{P}}^\perp\bigcap{\rm Ker\,}{\rm P\,}$, $d<\infty$. As \eqref{e-gue140330bII}, we have $f_j=Ff_j$, $j=1,2,\ldots,d$. Thus, 
$f_j\in C^\infty(X)$, $j=1,2,\ldots,d$, and hence $\hat{\mathcal{P}}^\perp\bigcap{\rm Ker\,}{\rm P\,}\subset C^\infty(X)$. 

We can repeat the procedure above and conclude that 
$\hat{\mathcal{P}_0}^\perp\bigcap{\rm Ker\,}{\rm P\,}\subset C^\infty(X)$, $\hat{\mathcal{P}_0}^\perp\bigcap\hat{\mathcal{P}}\subset C^\infty(X)$, $\hat{\mathcal{P}_0}^\perp\bigcap{\rm Ker\,}{\rm P\,}$, $\hat{\mathcal{P}_0}^\perp\bigcap\hat{\mathcal{P}}$ are all finite dimensional. 
\end{proof}

\section{Spectral theory for ${\rm P\,}$}\label{s-gue140413}

In this section, we will prove Theorem~\ref{t-gue140325I}. For any $\lambda>0$, put 
\[\Pi_{[-\lambda,\lambda]}:=E([-\lambda,\lambda]),\]
where $E$ denotes the spectral measure for ${\rm P\,}$ (see section 2 in Davies~\cite{Dav95}, for the precise meaning of spectral measure). We need 

\begin{thm}\label{t-gue140413}
Fix $\lambda>0$. We have ${\rm P\,}\Pi_{[-\lambda,\lambda]}\equiv0$ on $X$. 
\end{thm}

\begin{proof}
As before, let $N$ be the partial inverse of ${\rm P\,}$ and let $\Pi$ be the orthogonal projection onto ${\rm Ker\,}{\rm P\,}$. We have 
\begin{equation}\label{e-gue140413}
N{\rm P\,}+\Pi=I.
\end{equation}
From \eqref{e-gue140413}, we have 
\begin{equation}\label{e-gue140413I}
N{\rm P\,}^2\Pi_{[-\lambda,\lambda]}={\rm P\,}\Pi_{[-\lambda,\lambda]}.
\end{equation}
From \eqref{e-gue140325V}, \eqref{e-gue140413I} and notice that 
\[\mbox{${\rm P\,}^2\Pi_{[-\lambda,\lambda]}:L^2(X)\To L^2(X)$ is continuous},\]
we conclude that 
\begin{equation}\label{e-gue140413II}
\mbox{${\rm P\,}\Pi_{[-\lambda,\lambda]}:L^2(X)\To H^2(X)$ is continuous}.
\end{equation}
Similarly, we can repeat the procedure above and deduce that 
\begin{equation}\label{e-gue140413III}
\mbox{${\rm P\,}^2\Pi_{[-\lambda,\lambda]}:L^2(X)\To H^2(X)$ is continuous}.
\end{equation}
From \eqref{e-gue140413III}, \eqref{e-gue140413I} and \eqref{e-gue140325V}, we get 
\[
\mbox{${\rm P\,}\Pi_{[-\lambda,\lambda]}:L^2(X)\To H^4(X)$ is continuous}.
\]
Continuing in this way, we conclude that 
\begin{equation}\label{e-gue140413IV}
\mbox{${\rm P\,}\Pi_{[-\lambda,\lambda]}:L^2(X)\To H^m(X)$ is continuous, $\forall m\in\mathbb N_0$}.
\end{equation}
Note that ${\rm P\,}\Pi_{[-\lambda,\lambda]}=\Pi_{[-\lambda,\lambda]}{\rm P\,}=({\rm P\,}\Pi_{[-\lambda,\lambda]})^*$, where $({\rm P\,}\Pi_{[-\lambda,\lambda]})^*$ is the adjoint of ${\rm P\,}\Pi_{[-\lambda,\lambda]}$. By taking adjoint in \eqref{e-gue140413IV}, we get 
\[
\mbox{$\Pi_{[-\lambda,\lambda]}{\rm P\,}={\rm P\,}\Pi_{[-\lambda,\lambda]}:H^{-m}(X)\To L^2(X)$ is continuous, $\forall m\in\mathbb N_0$}.
\]
Hence, 
\begin{equation}\label{e-gue140413V}
\mbox{$({\rm P\,}\Pi_{[-\lambda,\lambda]})^2={\rm P\,}^2\Pi_{[-\lambda,\lambda]}:H^{-m}(X)\To H^m(X)$ is continuous, $\forall m\in\mathbb N_0$}.
\end{equation}
From \eqref{e-gue140413V}, \eqref{e-gue140413I} and \eqref{e-gue140325V}, the theorem follows. 
\end{proof} 

We need 

\begin{thm}\label{t-gue140413I}
For any $\lambda>0$, $\Pi_{[-\lambda,\lambda]}\equiv\Pi$ on $X$. 
\end{thm}

\begin{proof}
From \eqref{e-gue140413} and Theorem~\ref{t-gue140413}, we get 
\begin{equation}\label{e-gue140413VI}
\mbox{$\Pi\Pi_{[-\lambda,\lambda]}\equiv\Pi_{[-\lambda,\lambda]}$ on $X$}.
\end{equation}
On the other hand, it is clearly that $\Pi\Pi_{[-\lambda,\lambda]}=\Pi$. From this observation and \eqref{e-gue140413VI}, the theorem follows. 
\end{proof}

Now, we can prove 

\begin{thm}\label{t-gue140413II}
${\rm Spec\,}{\rm P\,}$ is a discrete subset in $\Real$ and for every $\lambda\in{\rm Spec\,}{\rm P\,}$, $\lambda\neq0$, $\lambda$ is an eigenvalue of ${\rm P\,}$ and the eigenspace 
\[H_\lambda({\rm P\,}):=\set{u\in{\rm Dom\,}{\rm P\,};\, {\rm P\,}u=\lambda u}\]
is a finite dimensional subspace of $C^\infty(X)$. 
\end{thm}

\begin{proof}
Since ${\rm P\,}$ has $L^2$ closed range, there is a $\mu>0$ such that ${\rm Spec\,}{\rm P\,}\subset]-\infty,-\mu]\bigcup[\mu,\infty[$. Fix $\lambda>\mu$. Put $\Pi_{[-\lambda,-\mu]\bigcup[\mu,\lambda]}:=E([-\lambda,-\mu]\bigcup[\mu,\lambda])$. Note that 
\[\Pi_{[-\lambda,-\mu]\bigcup[\mu,\lambda]}=\Pi_{[-\lambda,\lambda]}-\Pi_{[-\frac{\mu}{2},\frac{\mu}{2}]}.\]
From this observation and Theorem~\ref{t-gue140413I}, we see that 
\begin{equation}\label{e-gue140413VII}
\Pi_{[-\lambda,-\mu]\bigcup[\mu,\lambda]}\equiv0.
\end{equation}
We claim that ${\rm Spec\,}{\rm P\,}\bigcap\set{[-\lambda,-\mu]\bigcup[\mu,\lambda]}$ is discrete. If not, we can find $f_j\in{\rm Rang\,}E([-\lambda,-\mu]\bigcup[\mu,\lambda])$, $j=1,2,\ldots$, with $(\,f_j\,|\,f_k\,)=\delta_{j,k}$, $j,k=1,2,\ldots$. Note that 
\[f_j=\Pi_{[-\lambda,-\mu]\bigcup[\mu,\lambda]}f_j,\ \ j=1,2,\ldots.\]
From this observation, \eqref{e-gue140413VII} and Rellich's theorem, we can find subsequence $\set{f_{j_s}}^\infty_{s=1}$, $1\leq j_1<j_2<\cdots$, $f_{j_s}\To f$ in $L^2(X)$. Since $(\,f_j\,|\,f_k\,)=\delta_{j,k}$, $j,k=1,2,\ldots$, we get a contradiction. Thus, ${\rm Spec\,}{\rm P\,}\bigcap\set{[-\lambda,-\mu]\bigcup[\mu,\lambda]}$ is discrete. Hence ${\rm Spec\,}{\rm P\,}$ is a discrete subset in $\Real$. 

Let $r\in{\rm Spec\,}{\rm P\,}$, $r\neq0$. Since ${\rm Spec\,}{\rm P\,}$ is discrete, ${\rm P\,}-r$ has $L^2$ closed range. If ${\rm P\,}-r$ is injective, then ${\rm Range\,}({\rm P\,}-r)=L^2(X)$ and 
 \[({\rm P\,}-r)^{-1}:L^2(X)\To L^2(X)\]
 is continuous. We get a contradiction. Hence $r$ is an eigenvalue of ${\rm Spec\,}{\rm P\,}$. Put 
 \[H_r({\rm P\,}):=\set{u\in{\rm Dom\,}{\rm P\,};\, {\rm P\,}u=r u}.\]
We can repeat the procedure above and conclude that ${\rm dim\,}H_r({\rm P\,})<\infty$. Take $0<\mu_0<\lambda_0$ so that $r\in\set{[-\lambda_0,-\mu_0]\bigcup[\mu_0,\lambda_0]}$. From Theorem~\ref{t-gue140413I}, we see that 
\[\Pi_{[-\lambda_0,-\mu_0]\bigcup[\mu_0,\lambda_0]}\equiv0.\]
Since 
\[H_r({\rm P\,})=\set{\Pi_{[-\lambda_0,-\mu_0]\bigcup[\mu_0,\lambda_0]}f;\, f\in H_r({\rm P\,})},\]
$H_r({\rm P\,})\subset C^\infty(X)$. The theorem follows. 
\end{proof}

\noindent
{\small\emph{\textbf{Acknowledgements.} The author would like to express his gratitude to Jeffrey Case and Paul Yang for several useful conversations.}}

\end{document}